\documentclass[11pt, twoside]{amsart}

\title{A categorification of cyclotomic rings}
\date{\today}
\author{Robert Laugwitz}
\address{School of Mathematical Sciences,
University of Nottingham, Nottingham, NG7 2RD, UK}
\email{robert.laugwitz@nottingham.ac.uk}
\urladdr{https://www.nottingham.ac.uk/mathematics/people/robert.laugwitz}

\author{You Qi}
\address{Department of Mathematics,
University of Virginia, 141 Cabell Drive, Kerchof Hall, Charlottesville, VA 22904, USA}
\email{yq2dw@virginia.edu}
\urladdr{https://you-qi2121.github.io/mypage/}

\usepackage{mathabx}
\usepackage{import}
\usepackage{amsmath}
\usepackage{amsfonts,mathrsfs}
\usepackage{amsthm}
\usepackage{amssymb}
\usepackage[alphabetic, initials]{amsrefs}
\usepackage[english]{babel}
\usepackage{url}
\usepackage{fancyhdr}
\usepackage{graphicx}
\usepackage[
colorlinks=true,
linkcolor=black, 
anchorcolor=black,
citecolor=black, 
urlcolor=black, 
]{hyperref}
\usepackage{geometry}
\usepackage{comment}

\usepackage[all]{xy}

\makeatletter
\def\imod#1{\allowbreak\mkern10mu({\operator@font mod}\,\,#1)}
\makeatother

\newcommand{\isomorph}{\stackrel{\sim}{\longrightarrow}}

\newcommand{\op}[1]{\operatorname{#1}}
\newcommand{\oop}{\operatorname{op}}

\newcommand{\un}[1]{\underline{#1}}
\newcommand{\lmod}[1]{#1\text{-}\mathbf{gmod}}
\newcommand{\lomod}[1]{#1\text{-}\mathbf{mod}}
\newcommand{\lratmod}[1]{#1\text{-}\mathbf{rmod}}

\newcommand{\slmod}[1]{#1\text{-}\mathbf{g}\underline{\mathbf{mod}}}

\newcommand{\lcomod}[1]{#1\text{-}\mathbf{comod}}

\newcommand{\Ext}{\operatorname{Ext}}

\newcommand{\Hom}{\operatorname{Hom}}
\newcommand{\ide}{\operatorname{Id}}

\newcommand{\one}{\mathbf{1}}

\newcommand{\gvect}{\mathbf{gvec}}

\newcommand{\qvect}{\mathbf{gvec}_q}

\providecommand{\fr}[1]{\mathfrak{#1}}

\providecommand{\op}[1]{\operatorname{#1}}

\newcommand{\mZ}{\mathbb{Z}}
\newcommand{\mN}{\mathbb{N}}

\newcommand{\bfI}{\mathbf{I}}

\newcommand{\cB}{\mathcal{B}}

\newcommand{\cO}{\mathcal{O}}

\newcommand{\cP}{\mathcal{P}}

\def\d{{\operatorname{d}}}

\def\l{\ell}

\def\Z{{\mathbb{Z}}}

\newcommand{\shift}[1]{\{ #1 \}}
\newcommand{\Gro}[1]{\left[#1\right]}

\newtheoremstyle{mystyle}
  {0.5cm}                   
  {0.5cm}                   
  {\normalfont}           
  {}                      
  {\itfont\bfseries}  
  {:}                     
  {0.3cm}              
  {\thmname{#1}}

\newtheoremstyle{defstyle}
  {0.5cm}                   
  {0.5cm}                   
  {\normalfont}           
  {}     
  {\normalfont\bfseries}  
  {:}                     
  {0.3cm}              
  {\thmname{#1}\thmnumber{ #2}\thmnote{ (#3)}}

\numberwithin{equation}{section}
\newtheorem{theorem}{Theorem}[section]

\newtheorem{proposition}[theorem]{Proposition}
\newtheorem{corollary}[theorem]{Corollary}
\newtheorem{lemma}[theorem]{Lemma}

\newtheorem{theorem*}{Theorem}

\theoremstyle{definition}
\newtheorem{definition}[theorem]{Definition}
\newtheorem{notation}[theorem]{Notation}

\newtheorem{remark}[theorem]{Remark}
\newtheorem{example}[theorem]{Example}

\theoremstyle{remark}

\normalfont\upshape

\renewcommand{\sectionmark}[1]		
	{
	\markboth{\small\it \thesection{} #1}{}
	}

\begin{document}

\subjclass{(2020) Primary 18G90; Secondary 18G80, 16T05}
\keywords{Categorification, Hopfological algebra, cyclotomic rings, Hopf algebras, stable category}

\begin{abstract}
For any natural number $n\geq 2$, we construct a triangulated monoidal category whose Grothendieck ring is isomorphic to the ring  of cyclotomic integers $\mathbb{O}_n$. This construction provides an affirmative resolution to a problem raised by Khovanov in 2005.
\end{abstract}
\maketitle
\date{\today}


\section{Introduction}
\subsection{Backround} The seminal paper of Louis Crane and Igor B.~Frenkel \cite{CF} proposes that one should lift three-dimensional topological quantum field theories defined at a primitive $n$th root of unity to four-dimensional theories. The lifting process is usually referred to, in mathematics, as \emph{categorification}. One aim is to replace algebras appearing in the construction of the three-dimensional topological quantum field theories by categories, such that the original algebras can be recovered by passing to the Grothendieck group. However, a foundational obstacle to the program is the lack of a monoidal category that categorifies the cyclotomic ring of integers $\mathbb{O}_n$ at a primitive $n$th root of unity.

As an initial breakthrough, Khovanov \cite{Kh} observed that, when $n=p$ is a prime number, the graded Hopf algebra $H_p=\Bbbk[\d]/(\d^p)$ ($\mathrm{deg}(\d)=1$) over a field $\Bbbk$ of characteristic $p$ may be utilized to categorify $\mathbb{O}_p$. The basic idea is as follows. Inside the category of graded $H_p$-modules, the projective modules, which coincide with the injectives since $H_p$ is graded Frobenius (\cite{LS}), have their graded Euler characteristic equal to a multiple of that of the rank-one free module 
\begin{equation}
[H_p]=1+v+\cdots +v^{p-1}=\Phi_p(v).
\end{equation}
Here $\Phi_n(q)$ stands for the $n$th cyclotomic polynomial. Systematically killing projective-injective objects from $H_p$-modules results in a  triangulated monoidal category $\slmod{H_p}$ whose Gro\-then\-dieck ring is isomorphic to 
\begin{equation}
K_0(\slmod{H_p})= \dfrac{\mZ[v,v^{-1}]}{(\Phi_p(v))}\cong \mathbb{O}_p.
\end{equation}
The tensor triangulated category $\slmod{H_p}$ bears significant similarities with the usual homotopy category of abelian groups, and is thus also referred to as the \emph{homotopy category of $p$-complexes}.

The study of $\slmod{H_p}$ and algebra objects in these categories has been further developed in \cite{Qi}. The theory has since been applied to categorify various root-of-unity forms of quantum groups. We refer the reader to \cite{QS} for a brief summary and special phenomena at a $p$th root of unity.

\subsection{Outline of the construction} In this paper, we construct a triangulated monoidal category $\cO_n$ whose Grothendieck group is isomorphic to the ring of cyclotomic integers $\mathbb{O}_n$. We work in any characteristic, including characteristic zero, as long as the ground field contains a primitive $N$th root of unity, where $N=n^2/m$ and $m$ is the radical of $n$, the product of the distinct prime factors of $n$.  The construction is motivated from pioneering works of Kapranov \cite{Kap} and Sarkaria \cite{Sar} on $n$-complexes. When $n=p^a$ is a prime power, our work is equivalent to a graded version of $p$-complexes. With this approach, we remove the restriction on $n$ being a prime number.

Let us outline the construction. When $n=p^{a}$ is a prime power, one sees that $p$-complexes, up to homotopy, categorify $\mathbb{O}_{p^a}$ when the $p$-differential has degree $p^{a-1}$. The characteristic zero lift of the Hopf algebra $\Bbbk[\d]/(\d^p)$, which controls Kapranov--Sarkaria's $p$-complexes, is a Hopf algebra object in the braided monoidal category of $q$-graded vector spaces, where $q$ is a primitive $p^{2a-1}$th root of unity. 

Now suppose $n\geq 2$ is a general integer. Let $n=p_1^{a_1}\cdots p_t^{a_t}$ be the prime decomposition of $n$ and fix $q$, a primitive $N$th root of unity. By the one-factor case, it is natural to consider the Hopf algebra object $H_n$, in $q$-graded vector spaces, generated by commuting differentials $\d_1,\dots , \d_t$ subject to $\d_i^{p_i}=0$, $i=1,\dots, t$. The algebra $H_n$ is graded Frobenius, and thus has associated with it a well-behaved tensor triangulated stable category $\slmod{H_n}$. However, the Grothendieck ring of $\slmod{H_n}$ is defined by setting the character of the free module $H_n$ equal to zero. The last relation is usually larger than the cyclotomic relation $\Phi_n(v)=0$ which we would like to impose.

To obtain the correct relations in the Grothendieck ring, we use the elementary fact that $\Phi_n(v)$ occurs as the greatest common divisor of some prime power cyclotomic relations (Lemma \ref{cyclotomiclemma}). On the categorical level, the ideal generated by the greatest common divisor is categorified by a ``categorical ideal'', i.e, a thick triangulated subcategory $\underline{\bfI}$ inside $\slmod{H_n}$ which is closed under tensor product actions by $\slmod{H_n}$. Verdier localization at the categorical ideal $\underline{\bfI}$ yields the desired category $\mathcal{O}_n$, whose Grothendieck ring is defined by the desired relations and is isomorphic to $\mathbb{O}_n$.

Before moving on, let us also point out some connection to previous work. Furthering Kapranov and Sakaria, there is a significant amount of study on $n$-complexes in the literature. See, for instance, the lectures notes of Dubois-Violette \cite{DV} and the references therein. In \cite{Bic}, Bichon considers a Hopf algebra $A(q)=\Bbbk[x]/(x^n)\rtimes \Bbbk \mZ $ for which there is a monoidal equivalence between the category of $n$-complexes and $\lcomod{A(q)}$. One may similarly define a $\Z/n\Z$-graded version of Bichon's Hopf algebra, which resembles the classical Taft algebra. Recently, Mirmohades \cite{Mir} has introduced a tensor triangulated category arising from a suitable quotient of a tensor product of two Taft algebras. This category categorifies a primitive root of unity whose order is the product of two distinct odd primes. Our current work and \cites{Bic, Mir} are both unified under the frame of hopfological algebra \cite{Qi}.
 
 \subsection{Summary of contents} We now briefly describe the structure of the paper and summarize the contents of each section.
  
 In Section \ref{stablesect}, we review the construction of stable module categories in the particular case of finite-dimensional Hopf algebras. The Frobenius structure and the existence of an inner hom space allows an explicit identification of morphism spaces in the stable module categories.  

In Section \ref{sec-Hn}, we introduce the main object of study, a finite-dimensional braided Hopf algebra $H_n$, depending on a given natural number $n$. The category $\lmod{H_n}$ has a tensor product depending on a certain root of unity $q$. The braided Hopf algebra $H_n$ is primitively generated by certain commuting $p_k$-differentials $\d_k$, for $k=1,\ldots, t$. Using the Radford--Majid biproduct (or \emph{bosonization}, see \cites{Rad, Maj1}) of the differentials by the group algebra of a finite cyclic group, one obtains a related Hopf algebra, for which graded $H_n$-modules correspond to rational graded modules.
We also point out that $\lmod{H_n}$ has the structure of a spherical monoidal category in the sense of Barrett-Westbury \cite{BW}.

Next, we proceed, in Section \ref{sec-tensorideal}, to define a tensor triangulated ideal $\underline{\bfI}$ (Definition \ref{defnI}) in the (stable) module category of $H_n$. Upon factoring out the ideal by localization, we show, in Section \ref{sec-catfn}, that the quotient category has the desired Grothendieck ring $\mathbb{O}_n$ (Theorem \ref{thm-main}). The reason to introduce this ideal is as follows. On the Grothendieck ring level, we would like the objects of the ideal to have characters satisfying  cyclotomic relations dividing $q^n-1$ that are of lower order than the primitive condition $\Phi_n(q)=0$. Systematically killing these objects by taking a Verdier quotient in the stable category $\slmod{H_n}$ gives the lower order relations in the Grothendieck ring.
An example of such an object in the ideal is an $n$-complex which is freely generated by all but one of the differentials while the remaining differential acts by zero.
This objects has self-extensions and it is thus natural to require a filtration condition on the modules on the abelian level giving a triangulated tensor ideal $\bfI_k$. The bulk of the work in Section~\ref{sec-tensorideal} is devoted to showing that after passing to the stable category of $H_n$-modules, the ideals $\bfI_k$ are orthogonal so that their sum $\bfI$ is indeed closed under tensor product, extensions and direct summands. Now, the standard machinery of Verdier localization (quotient) can be used to obtain a triangulated quotient category $\mathcal{O}_n$ with a tensor product structure. Finally, in Theorem \ref{thm-main}, we prove that the Grothendieck ring of $\mathcal{O}_n$ is isomorphic to $\mathbb{O}_n$.

\subsection{Comparison and further directions} To conclude this introduction, let us make some comparison between our construction and the works \cites{Mir, Bic }, as well as indicate some further directions.

 We employ multiple nilpotent differentials $\d_1,\ldots, \d_t$ depending on the prime factors of $n$, in contrast to \cite{Mir}, thus getting rid of the restriction on $n$ having to be the product of two odd primes. In contrast to \cite{Mir} we only employ a single $\mZ$-grading rather than a bigrading. This requires us to use a filtration condition on modules in the ideal $\bfI_k$.

A negative result from \cite{Bic}*{Proposition 5} is the non-existence of a quasi-triangular structure on the Hopf algebra $A(q)$ describing $n$-complexes. In our setup, we show that instead of a quasi-triangular structure, there exist weak replacements given by functorial isomorphisms $V\otimes_q W\cong W\otimes_{q^{-1}} V$. For $n=2$, these satisfy the axioms of a braiding, but for other values of $n$ no analogue of the braiding axioms could be identified. We plan to explore this structure in subsequent works.

For further investigation, we would like to construct module categories over $\mathcal{O}_n$, developing triangulated analogues parallel to the  abelian theory of \cites{EGNO}. We will also seek interesting algebra objects in $\mathcal{O}_n$, in a similar way as done in \cites{KQ,EQ1} over the homotopy category of $p$-complexes. The Grothendieck groups of such algebra objects would then give rise to interesting modules over $\mathbb{O}_n$.
It would also be an interesting problem to combine the recent categorification of fractional integers due to Khovanov--Tian \cite{KT} in order to categorify the algebra $\mathbb{O}_n\left\lbrack\frac{1}{n}\right\rbrack$, over which the extended $3$-dimensional Witten--Reshetikhin--Turaev  TQFT lives.

\subsection{Acknowledgments} The authors would like to thank Igor Frenkel for his encouragements when the project started. They would also like to thank Mikhail Khovanov, Peter McNamara, Vanessa Miemietz, Joshua Sussan and Geordie Williamson for their interest in this work and helpful suggestions. Y.~Q.~is partially supported by the NSF grant DMS-1763328.

The authors further thank two anonymous referees for careful reading of this manuscript and helpful comments, in particular, pointing out a gap in a previous draft of this manuscript.

\section{The stable category}\label{stablesect}
\subsection{Notation}\label{notation-sect}
We start by fixing some conventions concerning $\mZ$-graded vector spaces over a ground field $\Bbbk$. 
Let us denote the category of finite-dimensional $\mZ$-graded vector spaces by $\gvect$. 

Let $M=\oplus_{i\in \mZ} M^i$ and $N=\oplus_{j\in \mZ}N^j$ be $\mZ$-graded vector spaces over  $\Bbbk$. We set $M\otimes_{\Bbbk} N$, or simply $M\otimes N$, to be the graded vector space
\[
M\otimes N:=\bigoplus_{k\in \mZ}(M\otimes N)^k,\quad \quad  (M\otimes N)^k:=\bigoplus_{i+j=k}M^i\otimes N^j.
\]
For any integer $k\in \Z$, we denote by $M\shift{k}$  the graded vector space $M$ with its grading shifted down by $k$:
$
(M\shift{k})^i=M^{i+k}.
$
The morphisms space $\Hom_\Bbbk^0(M,N)$ consists of homogeneous $\Bbbk$-linear maps from $M$ to $N$:
\[
\Hom_\Bbbk^0(M,N):=\left\lbrace f:M\longrightarrow N\middle| f(M^i)\subseteq N^i\right\rbrace.
\]
Writing $\Hom^i_\Bbbk (M,N):=\Hom^0_\Bbbk(M,N\shift{i})=\left\lbrace f:M\longrightarrow N\middle| f(M^j)\subseteq N^{i+j}\right\rbrace$, we set the graded hom space to be
\[
\Hom^\bullet_\Bbbk(M,N):=\bigoplus_{i\in \mZ} \Hom^i_\Bbbk (M,N).
\]
If no confusion can be caused, we will simplify $\Hom^\bullet_\Bbbk(M,N)$ to $\Hom^\bullet(M,N)$. 
A special case is the graded dual $M^*=\Hom^\bullet(M,\Bbbk)$.

Given three $\mZ$-graded vector spaces $M$, $L$ and $K$, the following easily proven tensor-hom adjunction will be used. There are isomorphisms of graded vector spaces, natural in $M,L, K$:
\begin{equation}\label{eqn-tensor-hom-vec-version}
\Phi\colon \Hom^\bullet(M\otimes L,K)\isomorph \Hom^\bullet(L,\Hom^\bullet(M,K)), \quad \quad \Phi(f)(l)(m):=f(m\otimes l),
\end{equation}
where $f\in \Hom^\bullet(M\otimes L, K)$, $m\in M$ and $l\in L$ are arbitrary elements.

We will also require (unbalanced) $q$-integers. In particular, for a formal variable $\nu$, we define polynomials 
\begin{align}
[n]_\nu=\frac{1-\nu^{n}}{1-\nu}=1+\nu+\ldots+\nu^{n-1}, && {n\brack k}_\nu=\frac{[n]_\nu!}{[k]_\nu![n-k]_\nu!}.
\end{align}
Given $q\in \Bbbk$, we set $[n]_q$ to be the value of $[n]_\nu$ evaluated at $\nu=q$.
For a $\mZ$-graded vector space $M$, denote by $$\dim_{\nu}(M)=\sum_{i\in \mZ} \dim_{\Bbbk}(M^i)\nu^i$$ the graded dimension of $M$. We abbreviate, for $f(\nu)=\sum_{i\in \mZ}{f_i\nu^i}\in \mN[\nu,\nu^{-1}]$, $$M^{f(\nu)}=\bigoplus_{i\in \mZ}M\shift{i}^{\oplus f_i}.$$

\subsection{Stable module categories}\label{stable-sect}
Let $H$ be a $\mZ$-graded self-injective algebra over a field $\Bbbk$. We denote by $\lmod{H}$ the category of finite-dimensional $\mZ$-graded modules over $H$, with morphisms of degree zero. For ease of notation, we will drop mentioning ``graded'' in what follows if no confusion can arise.

Note that, as $H$ is self-injective, a graded $H$-module is injective if and only if it is projective.
The (graded) stable category of finite-dimensional $H$-modules, denoted by $\slmod{H}$, is the categorical quotient of the category $\lmod{H}$ by the class of (graded) projective-injective objects. More precisely, recall that a degree-zero morphism is \emph{(homogeneous) null-homotopic} if it factors through a projective-injective $H$-module. For any two $H$-modules $M,N\in \lmod{H}$, let us denote the space of null-homotopic morphisms in $\lmod{H}$ by $\bfI_H^0(M,N)$. It is readily seen that, the collection of $\bfI_H^0(M,N)$'s ranging over all $M,N\in \lmod{H}$ constitute an ideal in $\lmod{H}$. Then $\slmod{H}$ has the same objects as $\lmod{H}$, and the morphism space between two objects $M,N\in \lmod{H}$ is by definition the quotient
\begin{equation}
\Hom_{\slmod{H}}(M,N):=\dfrac{\Hom_{\lmod{H}}^0(M,N)}{\bfI_H^0(M,N)}.
\end{equation}

It is a classical theorem that $\slmod{H}$ is triangulated, see \cite{Hap2}*{Theorem~9.4} and \cite{Hel}. The shift functor $[1]:\slmod{H}\longrightarrow\slmod{H}$ is defined as follows. For any $M\in \lmod{H}$, choose an injective envelope $I_M$ for $M$ in $\lmod{H}$ and let $K_M$ be the cokernel of the embedding map $\rho_M$:
\[
0\longrightarrow M\stackrel{\rho_M}{\longrightarrow} I_M\longrightarrow K_M\longrightarrow 0.
\]
Then $M[1]:= K_M$. The inverse functor $[-1]$ can be defined similarly by taking a projective cover and the corresponding kernel of the canonical epimorphism.

Let us also recall how distinguished triangles are defined in the stable category. Let $f\colon M\longrightarrow N$ be a morphism in $\lmod{H}$. Consider the diagram
\begin{equation}
\begin{gathered}
\xymatrix{
0 \ar[r] & M \ar[r]^{\rho_M} \ar[d]_f & I_M\ar[r]\ar[d] & M[1] \ar[r] \ar@{=}[d]& 0\\
0 \ar[r] & N \ar[r]^-u & C_f \ar[r]^-v & M[1]\ar[r] & 0
} 
\end{gathered} \ ,
\end{equation}
where the left-hand square is a push-out. One declares 
\begin{equation}\label{eqn-triangle}
M\stackrel{f}{\longrightarrow} N \stackrel{u}{\longrightarrow} C_f \stackrel{v}{\longrightarrow} M[1]
\end{equation}
to be a \emph{standard distinguished triangle}. Then any triangle in $\slmod{H}$ isomorphic to a standard one is called a \emph{distinguished triangle}.

We refer the reader to Happel's book \cite{Hap} for more details on this fundamental construction.

As for graded vector spaces, we set 
$$\Hom_{\lmod{H}}^i(M,N):=\Hom_{\lmod{H}}(M,N\shift{i}), \quad \quad \bfI_H^i(M,N):=\bfI^0_H(M,N\shift{i}),$$
and collect together
\begin{equation}
\Hom_{\slmod{H}}^\bullet(M,N):=\bigoplus_{i\in \mZ} \Hom_{\slmod{H}}(M,N\shift{i})=\bigoplus_{i\in \mZ}\left(\dfrac{\Hom_{\lmod{H}}^i(M,N)}{\bfI_H^i(M,N)}\right).
\end{equation}
Notice that this is different from the ext-space, which is denoted
\begin{equation}
\Ext^\bullet_{\slmod{H}}(M,N):=\bigoplus_{j\in \mZ}\Hom_{\slmod{H}}(M,N[j]).
\end{equation}

\subsection{Stable categories for finite-dimensional Hopf algebras}\label{stable-hopf-sect}
Now suppose $H$ is a finite-di\-men\-sion\-al graded Hopf algebra over $\Bbbk$. Our goal in this section is to provide a more explicit characterization of the morphism spaces in the graded stable module category $\slmod{H}$. The exposition here is a simplified version of the constructions in \cite{Qi}*{Section~5}.

Recall that a graded Hopf algebra $H$ is equipped with certain homogeneous structural maps called the \emph{counit} $\epsilon:H\longrightarrow \Bbbk$, the \emph{comultiplication} $\Delta:H\longrightarrow H\otimes H$, and the (invertible) \emph{antipode} $S:H\longrightarrow H^{\mathrm{op}}$, satisfying certain compatibility axioms with the algebra structure of $H$ (see, for instance, \cite{LS}). We will use adapted Sweedler's notation that
\begin{equation}
\Delta(h):=\sum_h h_1\otimes h_2.
\end{equation}
If $M$ and $N$ are $H$-modules, then $H$ acts on $M\otimes N$ by, for any $x\in M$, $y\in N$ and $h\in H$, 
\begin{equation}
h\cdot (x\otimes y)= \sum_h h_1x\otimes h_2y.
\end{equation}
We equip $M^*=\Hom^\bullet (M,\Bbbk)$ with the dual $H$-module structure
\begin{equation}
(h\cdot f)(x):=f(S^{-1}(h)x),
\end{equation}
for any $f\in M^*$ and $x\in M$. Notice that the grading on $M^*$ is given by 
$$(M^*)^k=\Hom_\Bbbk^0(M^{-k},\Bbbk).$$

More generally, let $M,N$ be two graded $H$-modules, we define a graded $H$-module structure on the $\Bbbk$-vector space $\Hom^\bullet_{\Bbbk}(M,N)$ by
\begin{equation}\label{internalhom}
(h\cdot f)(x)=\sum_h h_2f(S^{-1}(h_1)x).
\end{equation} 
It is easily checked that there is an isomorphism of graded $H$-modules
\begin{equation}\label{internalhom_dual}
\Hom^\bullet(M,N)\cong M^*\otimes N.  
\end{equation}
Furthermore, it is easy to check that the natural adjunction maps
\begin{equation}
\Bbbk\longrightarrow M^*\otimes M, \quad 1\mapsto \sum_{i}e_i^*\otimes e_i,
\end{equation}
\begin{equation}
M\otimes M^*\longrightarrow\Bbbk , \quad x\otimes f\mapsto f(x),
\end{equation}
commute with the $H$-actions, where $\{e_i\}$ is a homogeneous basis for $M$ and $\{e_i^*\}$ is the dual basis.

\begin{remark}\label{alternative-int-hom}
We remark that there is an alternative way to introduce internal homs for $\lmod{H}$, by using
$\Hom^\bullet(M,N)\cong N\otimes M^*$. In this case, the module structure is given by
$$(h\cdot f)(v)=\sum_h h_1f(S(h_2)v).$$
Note that under this action, $M^*$ is \emph{left} dual to $M$, whereas in equation (\ref{internalhom_dual}), $M^*$ plays the role of a \emph{right} dual. With the alternative convention, the modified form of the tensor-hom adjuction (c.f.~equation \eqref{eqn-tensor-hom-vec-version})
$$\Hom^\bullet(M\otimes L,N)\cong \Hom^\bullet(M,\Hom^{\bullet}(L,N))$$ 
is an isomorphism of $H$-modules.

The $H$-invariants discussed in Lemma \ref{lem-Hom-invariants} below will be naturally isomorphic for the two versions of internal homs.
\end{remark}

By a classical result of Larson--Sweedler \cite{LS}, $H$ is (graded) Frobenius, and, in particular, it is (graded) self-injective. Let $\Lambda$ be a fixed non-zero left integral in $H$, i.e., an element in $H$ such that for all $h\in H$, one has
\begin{equation}
h \Lambda = \epsilon(h)\Lambda.
\end{equation}
The element is unique up to a non-zero scalar, and hence a homogeneous element using that the multiplication on $H$ and $\epsilon$ are degree-preserving maps. Denote the degree of $\Lambda$ by $\mathrm{deg}(\Lambda):=\l$.
Then for any $H$-module $M$, we have a canonical embedding of $M$ into the injective $H$-module $M\otimes H$:
\begin{equation}
\rho_M: M\longrightarrow M\otimes H\shift{\l},\quad m\mapsto m\otimes \Lambda,
\end{equation}
because of the following result.

\begin{lemma}\label{projectivelemma}
Let $H$ be a (graded) Hopf algebra and $M$ a (graded) $H$-module. Then there is an isomorphism of tensor products of (graded) $H$-modules
\[
\phi_M: M\otimes H\cong M_0\otimes H,\quad m\otimes h\mapsto \sum_h S^{-1}(h_1)m\otimes h_2.
\]
Here $M_0$ stands for the vector space $M$ endowed with the trivial $H$-module structure
\[
hm_0=\epsilon(h)m_0,
\]
for any $h\in H$ and $m_0\in M_0$. In particular, $M\otimes H$ is projective and injective.
\end{lemma}
\begin{proof}
It is an easy exercise to check that the inverse of $\phi_M$ is given by
\[
\psi_M: M_0\otimes H\longrightarrow M\otimes H, \quad m_0\otimes h\mapsto \sum_h h_1m\otimes h_2.
\]
For the last statement, the projectivity of $M_0\otimes H$ is clear. For injectivity, one uses the well-known fact that a (possibly infinite) direct sum of injective $H$-modules remains injective if and only if $H$ is Noetherian.
\end{proof}

Despite the fact that the module $M\otimes H\shift{\l}$ is usually larger than the injective envelope of $M$, the functoriality of this canonical map in $M$ will allow us to understand the morphism spaces in the stable category more explicitly.

Recall that, for any $H$-module $M$, the space $M^H$ of \emph{$H$-invariants in $M$} consists of
\begin{equation}
M^H:=\{m\in M|hm=\epsilon(h)m,~\forall h \in H\}.
\end{equation}

\begin{lemma}\label{lem-Hom-invariants}
The space of $H$-invariants in $\Hom^\bullet(M,N)$ coincides with the space of $H$-module maps between $M$ and $N$:
\[
\Hom^\bullet(M,N)^H=\Hom_H^\bullet(M,N).
\]
In particular, there is an identification $\Hom^0(M,N)^H=\Hom_H^0(M,N)$.
\end{lemma}
\begin{proof}
If $f$ is an $H$-linear map, it is clear that, for any $h\in H$ and $m\in M$, 
\[
(h\cdot f)(m)=\sum_h h_2f(S^{-1}(h_1)m)=\sum_h f(h_2S^{-1}(h_1)m)=\epsilon(h)f(m),
\]
so that $h\cdot f=\epsilon(h) f$. Here, in the last equality, we have used the fact that, for any element $h\in H$, the identity $\sum_h S^{-1}(h_2)h_1=\epsilon(h)$ holds. 

On the other hand, if $f\in \Hom_\Bbbk(M,N)^H$, then
\[
h(f(m))=\sum_h h_3f(S^{-1}(h_2)h_1m)=\sum_h(h_2\cdot f)(h_1m)=\sum_h\epsilon(h_2)f(h_1m)=f(hm).
\]
The lemma now follows.
\end{proof}

The lemma can be rephrased as saying that the category $\lmod{H}$ is an \emph{enriched category} over itself.

\begin{lemma}\label{lem-canonical-factorization}
An $H$-module homomorphism $f:M\longrightarrow N$ factors through a projective-injective $H$-module if and only if there is an $H$-module map $g$ making the following diagram commute:
\[
\begin{gathered}
\xymatrix{
M \ar[dr]_{\rho_M} \ar[rr]^{f}&& N\\
& M\otimes H\shift{\l}\ar[ur]_g &
}
\end{gathered} \ .
\]
\end{lemma}
\begin{proof}
It suffices to prove the result when $N$ is projective-injective. In this case, consider the following commutative diagram.
\[
\begin{gathered}
\xymatrix{
M\ar[r]^f \ar[d]_{\rho_M}& N \ar[d]^{\rho_N} \\
M\otimes H\shift{\l}\ar[r]^{f\otimes \mathrm{Id}_H} & N\otimes H\shift{\l}
}
\end{gathered} \ .
\]
Since both $N$ and $N\otimes H\cong H^{\dim_\nu(N)}$ are injective, and $\rho_N=\mathrm{Id}_N\otimes \Lambda$ is an embedding, there is an $H$-module splitting map $g^\prime \colon N\otimes H\shift{\l}\longrightarrow N$ such that $g^\prime \circ \rho_N=\mathrm{Id}_N$. Now the lemma follows by taking $g=g^\prime \circ (f\otimes \mathrm{Id}_H)$.
\end{proof}

\begin{lemma}\label{lem-null-homotopy-integral}
A degree-zero $H$-module map $f\colon M\longrightarrow N$ factors through the canonical injective map $\rho_M\colon M\longrightarrow M\otimes H\shift{\l}$ if and only if there is $\Bbbk$-linear map  $g\colon M\longrightarrow N$ of degree $-\l$ such that
\[
f(m)= (\Lambda\cdot g)(m) =\sum_\Lambda \Lambda_{2}g(S^{-1}(\Lambda_{1}) m)
\]
for any $m\in M$.
\end{lemma}
\begin{proof}
If $f=\Lambda\cdot g$ for some $\Bbbk$-linear $g:M\longrightarrow N$, we will extend $g$ to an $H$-linear map 
$$
\widehat{g}:M\otimes H\longrightarrow N,\quad
\widehat{g}(m\otimes h):= (h\cdot g)(m)= \sum_h h_{2}g(S^{-1}(h_{1}) m).
$$
It will then follow by construction that $f=\widehat{g}\circ \rho_M$. Indeed, we check that $\widehat{g}$ is $H$-linear. For any $x,h\in H$ and $m\in M$, we have that
\begin{align*}
\widehat{g}(x\cdot(m\otimes h)) & = \sum_x \widehat{g}(x_{1}m\otimes x_{2}h) = \sum_{x,h} (x_{2}h)_{2} g(S^{-1}((x_{2}h)_{1})x_{1}m)\\
& = \sum_{x,h} x_{3}h_{2}g(S^{-1}(h_{1})S^{-1}(x_{2})x_{1}m)=\sum_{x,h}x_{2}h_{2}g(S^{-1}(h_{1})\epsilon(x_{1})m)\\
& = \sum_h xh_2g(S^{-1}(h_1)m) =x \widehat{g}(m\otimes h).
\end{align*}
Here in the fourth equality, we have used the fact that, for any element $x\in H$, it holds that $\sum_x S^{-1}(x_2)x_1=\epsilon(x)$.

Conversely, if $f$ factors as a composition of $H$-linear maps
\[
f: M\stackrel{\rho_M}{\longrightarrow}M\otimes H \stackrel{\widehat{g}}{\longrightarrow} N,
\]
so that $f(m)=\widehat{g}(m\otimes \Lambda)$ for any $m\in M$, we then define a $\Bbbk$-linear map $g:M\longrightarrow N$ by $g(m):=\widehat{g}(m\otimes 1)$. It remains to verify that $\Lambda\cdot g=f$. To do this, we compute, for any $m\in M$, 
\begin{align*}
(\Lambda \cdot g)(m) & = \sum_\Lambda \Lambda_2 g(S^{-1}(\Lambda_1)m)=\sum_\Lambda \Lambda_2 \widehat{g}(S^{-1}(\Lambda_1)m\otimes 1)= \sum_\Lambda \widehat{g}(\Lambda_2(S^{-1}(\Lambda_1)m\otimes 1)) \\
& = \sum_\Lambda \widehat{g}(\Lambda_2 S^{-1}(\Lambda_1)m\otimes \Lambda_3) =\sum_\Lambda \widehat{g}(\epsilon(\Lambda_1)m\otimes \Lambda_2) = \widehat{g}(m\otimes \Lambda)\\
& = \widehat{g}\circ \rho_M (m) =f(m).
\end{align*}
The result follows.
\end{proof}

\begin{theorem}\label{thm-hom-as-stable-invariant}
Let $H$ be a finite-dimensional graded Hopf algebra with a non-zero left integral $\Lambda\in H$. For any $H$-modules $M$, $N$, there is a canonical isomorphism 
\[
\Hom_{\slmod{H}}^\bullet(M,N) =\dfrac{\Hom^\bullet(M,N)^H}{\Lambda\cdot \Hom^{\bullet-\l}(M,N)},
\]
which is natural in both $M$ and $N$.
\end{theorem}
\begin{proof}
It suffices to show the statement in degree zero. By Lemma \ref{lem-Hom-invariants}, the numerator in the equality above coincides with the space of $H$-intertwining maps. Combining Lemma \ref{lem-canonical-factorization} and \ref{lem-null-homotopy-integral}, one sees that the space of maps between two $H$-modules that factor through projective-injective modules coincides with $\bfI_H^\bullet(M,N)\cong\Lambda\cdot \Hom^{\bullet-\l}(M,N)$. The theorem follows.
\end{proof}

The theorem implies that the stable category $\slmod{H}$ for a finite-dimensional Hopf algebra is equipped with an \emph{internal Hom}, which is no other than the space of graded vector space homomorphisms $\Hom^\bullet(M,N)$. 

\begin{corollary}\label{eqn-tensor-hom-adj}
The graded tensor-hom adjunction holds in $\slmod{H}$. That is, for any $M$, $N$ and $L$ in $\slmod{H}$, there is an isomorphism of graded vector spaces
\begin{equation*}
\Hom_{\slmod{H}}^\bullet(M\otimes L, N)\cong  \Hom_{\slmod{H}}^\bullet(L,\Hom^\bullet(M,N)).
\end{equation*}
In particular, there is an isomorphism of ungraded vector spaces
\[
\Hom_{\slmod{H}}(M,N)\cong \Hom_{\slmod{H}}(\Bbbk,\Hom^\bullet(M,N))
\]
functorial in $M$ and $N$.
\end{corollary}
\begin{proof}
This follows from taking the canonical isomorphism of $H$-modules (upgraded from the vector space version \eqref{eqn-tensor-hom-vec-version})
\[
\Phi\colon \Hom^\bullet(M\otimes L, N)\isomorph \Hom^\bullet(L,\Hom^\bullet(M,N)),\quad  \Phi(f)(l)(m):=f(m\otimes l),
\]
and applying the theorem to both sides.

The second equation is then established by taking $L=\Bbbk$ and taking degree zero parts on both sides in the first equation.
\end{proof}

\begin{remark}
We will be applying the results in this section to a slightly more general situation than graded Hopf algebras in what follows. In particular, we will be studying graded vector spaces with a non-trivial braiding, and $H$ being a Hopf algebra object in this braided category. The results of this section hold without any changes as long as $H$ is also a Frobenius algebra object.
\end{remark}

\section{The Hopf algebra \texorpdfstring{$H_n$}{Hn} and its bosonization}\label{sec-Hn}
\subsection{Braided vector spaces}
The category $\gvect$ of finite-dimensional $\mZ$-graded vector spaces is naturally a symmetric monoidal category with the symmetric braiding $\tau(v\otimes w)=w\otimes v$. For the purpose of this paper, we will consider a non-symmetric braiding on this category. 

Fix a natural number $N\geq 2$ and let $\Bbbk$ be a field of any characteristic which contains a primitive $N$th root of unity $q$. Given two graded vector spaces $V, W$ define the $\mZ$-linear map $\Psi_{V,W}\colon V\otimes W\to W\otimes V$ determined by
\begin{align}
\Psi_{V,W}(v\otimes w)=q^{\deg(v)\deg(w)}w\otimes v, 
\end{align}
where $v,w$ are homogeneous elements. It follows that $\Psi$ defines a braiding on the category of $\mZ$-graded vector spaces. We denote the braided monoidal category thus obtained by $\qvect$ (in contrast to the symmetric monoidal category $\gvect$).

%

Via a form of Tannakian reconstruction, the category $\gvect$ is equivalent to the category of finite-dimensional comodules over the group algebra $\Bbbk C$, where $C=\langle K\rangle$ is the free abelian group generated by $K$. The Hopf algebra $\Bbbk C$ can be equipped with a dual $R$-matrix $R\colon \Bbbk C\otimes \Bbbk C\to \Bbbk$ defined by
$$R(K^i\otimes K^j)=q^{ij},$$
see e.g. \cite{Maj1}*{Example 2.2.5}.
We denote the category of finite-dimensional $C$-comodules with braiding obtained from $R$ by $\lcomod{C}_q$. Hence there is an equivalence of braided monoidal categories
$$\lcomod{\Bbbk C}_q\simeq \qvect.$$

\subsection{Graded rational modules}\label{ratgraded-sect}
Let $H$ be a Hopf algebra object in $\qvect$. We want to study the category of $H$-modules in $\qvect$ in terms of graded modules over a $\Bbbk$-Hopf algebra. For this, we first pass from $\qvect$ to a braided category of modules over the group algebra of a finite cyclic group. 

Let $C_N$ denote the finite group $C/(K^N)$ and let $\pi_N\colon C\to C_N$ be the canonical quotient homomorphism of groups. Then there is an induced Hopf algebra morphism $\Bbbk C\to \Bbbk C_N$, which, in turn, produces a functor of monoidal categories
\begin{align*}
(\pi_n)_\ast&\colon \lcomod{\Bbbk C}\longrightarrow \lcomod{\Bbbk C_N}, & (V,\delta)\longmapsto (V, (\pi_N\otimes \ide_V)\delta),
\end{align*}
where $\delta$ denotes the left coaction on $V$. The dual $R$-matrix $R$ on $\Bbbk C$ induces a dual $R$-matrix on $\Bbbk C_N$ so that $(\pi_N)_\ast$ becomes a functor of braided monoidal categories
\begin{align*}
(\pi_N)_\ast&\colon \lcomod{\Bbbk C}_q\longrightarrow \lcomod{\Bbbk C_N}_q.
\end{align*}

For the next result, note that $N$ must be invertible in $\Bbbk$ since, as the polynomial $f(x)=x^N-1$ does not have multiple roots in $\Bbbk$, its formal derivative equals $Nx^{N-1}\neq 0$.  

\begin{proposition}
There is an equivalence of braided monoidal categories $\lcomod{\Bbbk C_N}_q\simeq \lomod{\Bbbk C_N}_q$. Here, the latter is the braided monoidal category of $\Bbbk C_N$-modules with braiding given by the $R$-matrix
\begin{equation}\label{Rmatrixgroup}
R=\frac{1}{N}\sum_{i,j} q^{-ij}K^i\otimes K^j.
\end{equation}
\end{proposition}
\begin{proof}
Denote by $\Bbbk [C_N]$ the algebra of $\Bbbk$-linear functions $C_N\to \Bbbk$. This is a Hopf algebra, dual to the group algebra $\Bbbk C_N$. Consider the basis $\lbrace \delta_i\mid 0\leq i\leq N-1\rbrace$ for $\Bbbk[C_N]$, where $\delta_i(K^j)=\delta_{i,j}$; we also denote $\delta_{k}=\delta_{l}$ if $k=l \mod N$. The relations, and structural morphisms $\Delta$, $\epsilon$, and $S$ of the Hopf algebra structure for $\Bbbk[C_N]$, are given by
\begin{align}
\delta_i\delta_j&=\delta_{i,j}\delta_i,& 1&=\sum_{i=0}^{N-1}\delta_i, &\Delta(\delta_i)&=\sum_{a+b=i}\delta_a\otimes \delta_b, &\epsilon(\delta_i)&=\delta_{i,0},& S(\delta_i)=\delta_{-i}.
\end{align}
An explicit Hopf algebra pairing $(~,~)\colon  \Bbbk[C_N]\otimes\Bbbk C_N \to \Bbbk$ is given by $(\delta_i, K^j)=\delta_{i,j}$. This non-degenerate Hopf algebra pairing defines, as $\Bbbk [C_N]$ is finite-dimensional and (co)commutative, an equivalence of monoidal categories
$$
\lcomod{\Bbbk C_N}\simeq \lomod{\Bbbk[C_N]}, 
$$
where for a homogeneous element $v$ of degree $i$  we define the action $\delta_j\cdot v=\delta_{i,j}v$. 

Under the pairing $(~,~)$ the dual $R$-matrix $R(K^i,K^j)=q^{ij}$ for the group algebra $\Bbbk C_N$ induces on $\Bbbk[C_N]$ the universal $R$-matrix 
\begin{equation}\label{Rmatrix}
R=\sum_{i,j}q^{ij}\delta_i\otimes \delta_j.
\end{equation}
Hence, denoting the obtained braided monoidal category of $\Bbbk[C_N]$-module by $\lomod{\Bbbk[C_N]}_q$, we obtain an equivalence of braided monoidal categories $\lcomod{\Bbbk C_N}_q\simeq \lomod{\Bbbk[C_N]}_q$.

Note also that, since the polynomial $f(x)=x^N-1$ splits over $\Bbbk$, $\Bbbk[C_N]$ is isomorphic to $\Bbbk C_N$ as a Hopf algebra, although not canonically. An isomorphism $\Bbbk C_N\to \Bbbk [C_N]$ is given by sending $K$ to the group like element $\sum_{i}q^{i}\delta_i$. Since $\delta_i$, $i=0,\dots, N$ are mutually orthogonal idempotents, one has $\left(\sum_{i}q^{i}\delta_i\right)^k=\sum_{i}q^{ik}\delta_i$. The inverse is given by sending $\delta_j$ to $\frac{1}{N}\sum_iq^{-ij}K^i$. 
The above isomorphism of Hopf algebras $\Bbbk C_N\cong \Bbbk[C_N]$ makes $\Bbbk C_N$ a quasi-triangular Hopf algebra with universal $R$-matrix given as in equation (\ref{Rmatrixgroup}).  Indeed, we compute that applying the above isomorphism to the universal $R$-matrix of $\Bbbk C_N$ from equation \eqref{Rmatrixgroup} gives 
\begin{align*}
\frac{1}{N}\sum_{i,j} q^{-ij}\sum_{a,b}q^{ia+jb}\delta_a\otimes \delta_b&=\frac{1}{N}\sum_{i,j,a,b} q^{ab}q^{-(a-i)(b-j)}\delta_a\otimes \delta_b\\
&=\frac{1}{N}\sum_{a,b} q^{ab}\delta_a\otimes \delta_b\sum_{i,j}q^{-(a-i)(b-j)}\\
&=\sum_{a,b} q^{ab}\delta_a\otimes \delta_b,
\end{align*}
which is the universal $R$-matrix of $\Bbbk[C_N]$ from equation \eqref{Rmatrix}. See \cite{Maj1}*{Example 2.1.6} for a direct proof of this quasi-triangular Hopf algebra structure.
\end{proof}

The convolution inverse $R^{-\ast}$ is given by 
\begin{equation}
R^{-\ast}=(S\otimes \ide)R=\frac{1}{N}\sum_{i,j}q^{ij}K^i\otimes K^j.
\end{equation}

In any braided monoidal category $\cB$, we can form the braided tensor product $D_1\otimes D_2$ of two algebra objects $D_1, D_2$ in $\cB$. The product $m_{D_1\otimes D_2}$ is given by
\[
m_{D_1\otimes D_2}=(m_{D_1}\otimes m_{D_2})(\ide_{D_1}\otimes \Psi_{D_2,D_1}\otimes \ide_{D_2}).
\]
Tensor products of coalgebra objects are defined similarly.
We can also define bialgebra (or Hopf algebra objects) in $\cB$. These are sometimes called \emph{braided Hopf algebras}, see e.g. \cite{Maj1}*{Definition 9.4.5}. The crucial point is that a bialgebra $B$ in $\cB$ is both an algebra and coalgebra in $\cB$ such that $\Delta$ and $\epsilon$ are morphisms of algebras, i.e.  
\begin{gather}\label{braidedhopf1} 
\begin{split}\Delta_B \circ m_B=(m_B\otimes m_B)\circ (\ide_B\otimes \Psi_{B,B}\otimes \ide_B)\circ (\Delta_B\otimes \Delta_B), \\
\Delta_B\circ 1_B=1_B\otimes 1_B,\qquad
\epsilon\circ m =\epsilon\otimes \epsilon, \qquad 1\circ \epsilon=\ide.\end{split}
\end{gather}

Let $H$ be a braided Hopf algebra in $\qvect$. Then the image of $H$ under the composite functor
$$\cP\colon \lcomod{C}_q\longrightarrow \lcomod{C_N}_q\isomorph \lomod{\Bbbk C_N}_q,$$
is a braided Hopf algebra in $\lomod{\Bbbk C_N}_q$. By slight abuse of notation, this image is also denoted by $H$. We may now consider the Radford--Majid biproduct (\cite{Rad}, also called the bosonization \cite{Maj1}*{Theorem 9.4.12}) $H\rtimes \Bbbk C_N$. By construction, there is an equivalence of categories
$$\lomod{H\rtimes \Bbbk C_N}\cong \lomod{H}~(\lmod{\Bbbk C_N}_q),$$
where the latter denotes the category of modules over $H$ within the braided monoidal category $\lomod{\Bbbk C_N}_q$. That is, the morphisms of the $H$-module structure are all morphisms in this category, cf. \cite{Maj1}*{Section~9.4}. The monoidal functor $\cP$ therefore restricts to a monoidal functor
$$\cP_H\colon \lomod{H}(\qvect)\longrightarrow \lomod{H\rtimes \Bbbk C_N}.$$

Note that, in addition, $H$ is a graded $\Bbbk$-algebra, and the bosonization $H\rtimes \Bbbk C_N$ is a graded Hopf algebra, where $\deg K=0$. Thus, we can consider graded modules over $H\rtimes \Bbbk C_N$, and the essential image of the functor $\cP_H$ is contained in $\lmod{H\rtimes \Bbbk C_N}$.

\begin{definition}\label{rational-graded-def}
A graded $H\rtimes \Bbbk C_N$-module $V=\bigoplus_{i\in \mZ} V^i$ is a \emph{rational graded module} if  for any $v\in V^i$, $K\cdot v=q^i v$. 

We denote the category of rational graded $H\rtimes \Bbbk C_N$-modules, together with morphisms of graded $H\rtimes \Bbbk C_N$-modules, by $\lratmod{H\rtimes \Bbbk C_N}$.
\end{definition}

Working with rational graded modules we obtain a characterization of the braided monoidal category $\lmod{H}$ $(\lmod{\Bbbk C}_q)$ in terms of modules over the finite-dimensional Hopf algebra $H\rtimes \Bbbk C_N$:

\begin{proposition}\label{rational-graded-prop}
An $H\rtimes \Bbbk C_N$-module $V$ is in the essential image of the monoidal functor $\cP_H$ if and only if $V$ is a rational graded module.
\end{proposition}
\begin{proof}
Let $V$ be an $H\rtimes \Bbbk C_N$-module in the essential image of $\cP_H$. Then, in particular, $V$ is a graded $H\rtimes \Bbbk C_N$-module. For a vector $v\in V^i$ we have that 
$$K\cdot v=\left(\sum_i q^i \delta_i\right)\cdot v=q^i v.$$
Hence, $V$ is a rational graded module. Conversely, let $W$ be a rational graded module over $H\rtimes \Bbbk C_N$. Then $W$ is graded, and hence a $\Bbbk C$-comodule. Using that $H\hookrightarrow H\rtimes \Bbbk C_N$ is a graded subalgebra, $W$ becomes a graded $H$-module, denoted by $W'$. We have to show that $\cP_H(W')$ and $W$ are isomorphic as graded $H\rtimes \Bbbk C_N$-modules. By construction, they are the same graded $H$-modules, and for a vector $w\in \cP_H(W'^i)$, $K\cdot w=q^i\cdot w$. As $W$ is rational graded, the same formula describes the $C_N$-action on $W$. It follows that $\cP_H(W')$ and $W$ are also isomorphic as  $H\rtimes \Bbbk C_N$-modules.
\end{proof}

It follows that, as a full subcategory of $\lmod{H\rtimes \Bbbk C_N}$, $\lmod{H}$ is closed under tensor products and extension. As all rational $C_N$-modules are graded modules, all constructions from Section \ref{notation-sect} can be applied to rational graded $C_N$-modules. In particular, the internal graded hom $\Hom^\bullet(M,N)$ of two rational graded modules is itself a rational graded module.

\begin{notation}
This section shows that the category $\lmod{H}$ of graded $H$-modules has a tensor product which can either be computed using the coproduct within $\qvect$ or, equivalently, the coproduct of the bosonization by Proposition \ref{rational-graded-prop}. In Section \ref{sec-tensorideal}, we will simply denote the resulting monoidal category by $\lmod{H}$.
\end{notation}

\subsection{A braided Hopf algebra}
We first fix some notation and assumptions. Let $n\geq 2$ be a positive integer, and factorize $n=p_1^{a_1}\ldots p_t^{a_t}$ as a product of distinct prime powers. Denote by $m=p_1\ldots p_t$ the radical of $n$ and define $N:=n^2/m$. Set $n_k:=n/p_k$, $m_k:=m/p_k$.

We assume the ground field $\Bbbk$ contains a primitive $N$th root of unity $q$. Then we denote $\xi:=q^{n/m}$, which is a primitive $n$th root of unity, and $\xi_k:=\xi^{m_k}=q^{n_k}$.

\begin{definition}
Let $H_n$ be the $\Bbbk$-algebra
\[
H_n:=\dfrac{\Bbbk[\d_1,\dots, \d_t]}{(\d_1^{p_1},\dots, \d_t^{p_t})},
\]
which is graded by setting $\mathrm{deg}(\d_k)=n_k$ for all $1\leq k \leq t$. 
\end{definition}


\begin{lemma}\label{lem-Hn-Frobenius-Hopf}
The algebra $H_n$ is Frobenius with a non-degenerate trace pairing given on basis elements by
\[
\mathrm{Tr}(\d_1^{a_1}\cdots \d_t^{a_t})=
\left\{
\begin{array}{ccc}
1 && (a_1,\dots, a_t)=(p_1-1,\dots , p_t-1),\\
0 && \textrm{otherwise.}
\end{array}
\right.
\]
\end{lemma}

Define the \emph{comultiplication} map  $\Delta\colon H_n\longrightarrow H_n\otimes H_n$ on generators by
\begin{equation}
\Delta(\d_k):=\d_k\otimes 1+1\otimes \d_k,
\end{equation}
and set the \emph{counit} and \emph{antipode} maps to be
\begin{equation}
\epsilon\colon H_n\longrightarrow \Bbbk,\quad \quad \epsilon (\d_k)=0,
\end{equation}
\begin{equation}
S\colon H_n\longrightarrow H_n^{\mathrm{op}},\quad \quad S(\d_k)=-\d_k,
\end{equation}
for all $1\leq k \leq t$.

\begin{lemma}
The above definitions of $\Delta$, $\epsilon$, and $S$ uniquely extend to give $H_n$ the structure of a primitively generated Hopf algebra object in the braided category $\qvect$ of $q$-vector spaces.
\end{lemma}
\begin{proof}
It is well-known that the free $\Bbbk$-algebra $\Bbbk\langle \d_1,\ldots, \d_t\rangle$ extends to give the structure of a primitively generated braided Hopf algebra in the braided category of $q$-vector spaces in a unique way. The conditions from equation (\ref{braidedhopf1}) inductively imply that
\begin{gather}
\Delta(\d_k^a)=\sum_{i=0}^{a}{a \brack i}_{\xi_k^{n_k}} \d_k^{i}\otimes \d_k^{a-i}\\
\epsilon(\d_k^a)=\delta_{a,0}, \qquad \qquad
S(\d_k^a)=(-1)^a \xi_k^{a(a-1)n_k/2}\d_k^a.
\end{gather}
It hence remains to check that the ideal generated by  $[\d_k,\d_l]$ for $l\neq k$ and $\d_k^{p_k}$ is a Hopf ideal. This follows as the generators are primitive elements:
\begin{equation*}
\Delta([\d_k,\d_l])=[\d_k,\d_l]\otimes 1+1\otimes [\d_k,\d_l],
\end{equation*}
\begin{equation*}
\Delta(\d_k)^{p_k}=(\d_k\otimes 1+1\otimes \d_k)^{p_k}=\sum_{i=0}^{p_k}{p_k \brack i}_{\xi_k^{n_k}} \d_k^{i}\otimes \d_k^{p_k-i}=\d_k^{p_k}\otimes 1+ 1\otimes \d_k^{p_k}.\qedhere
\end{equation*}
Here, we have used that $\xi_l^{n_k}=\xi^{m_ln_k}=q^{n_kn_l}=1$, and that $\xi_k^{n_k}=\xi^{m_kn_k}=q^{n_k^2}$ is a primitive $p_k$-th root of unity.
\end{proof}

\begin{remark}\label{nichols}
The braided Hopf algebra $H_n$ can be constructed as the Nichols algebra over the Yetter--Drinfeld module $V=\op{Span}_\Bbbk\lbrace \d_1,\ldots, \d_t\rbrace$ over the group $C_N$ (see e.g. \cite{AS} for this construction). 
The $C_N$-coaction $\delta$ on $V$ is given by $\delta(\d_k)=K^{n_k}\otimes \d_k$, and the $C_N$-action is given by $K\cdot \d_k=\xi_k\d_k$.
The Yetter--Drinfeld braiding $\Psi_V$ of $V$ determines the relations in the Nichols algebra $\cB(V)=H_n$. Note that, for distinct indices $k, l=1,\ldots, t$,
$$\Psi(\d_k\otimes \d_l)=\xi_l^{n_k}\d_{l}\otimes \d_k=\d_{l}\otimes \d_k$$
as $p_l$ divides $n_k$. This implies that in the Nichols algebra $H_n$ the relations $[\d_k,\d_l]=0$ hold. Further, $\Psi(\d_k\otimes \d_k)=\xi_k^{n_k}\d_k\otimes \d_k$. Using that $\xi_k^{n_k}$ is a primitive $p_k$th root of unity in $\Bbbk$, this computation of the braiding implies that in the Nichols algebra, $\d_k^{p_k}=0$. These are the only relations (cf. \cite{AS}*{Theorem~4.3}).
This construction as a Nichols algebra proves that $H_n$ is a braided Hopf algebra in $\lmod{\Bbbk C_N}_q$ which is generated by primitive elements. 

This construction of $H_n$ further implies that $H_n$ is self-dual as a braided Hopf algebra. That is, there is a non-degenerate Hopf pairing $\langle \mbox{-},\mbox{-}\rangle\colon H_n\otimes H_n\to \Bbbk$, defined on generators by
$\langle \d_k,\d_l\rangle=\delta_{k,l}$ in the category $\qvect$ (see \cite{Lus}*{Proposition~1.2.3}).

By construction, $H_n$ is a commutative algebra. Note that, even though $\Psi\Delta(\d_k)=\Delta(\d_k)$ for all generators,  $H_n$ is \emph{not} braided cocommutative in $\qvect$. This follows using \cite{Sch}*{Corollary 5}, since $\Psi_{H_n,H_n}^2\neq \ide\otimes \ide$.
\end{remark}

\begin{remark}
The element $\Lambda:=\d_1^{p_1-1}\cdots \d_t^{p_t-1}
$ has the property that
\begin{align*}
h\Lambda=\epsilon(h)\Lambda,&& \forall h\in H_n.
\end{align*}
That is, $\Lambda$ is an integral element for the braided Hopf algebra $H_n$ (as in \cite{BKLT}*{Definition 3.1}), cf. also Lemma \ref{integrals} below. Note that
\begin{align} 
\mathrm{Tr}(h)=\langle h,\Lambda\rangle,&&\forall h\in H_n,
\end{align}
with respect to the integral and trace map from Lemma \ref{lem-Hn-Frobenius-Hopf}.
We denote the degree of the integral $\Lambda$ by 
\begin{equation}
\l:=\mathrm{deg}(\Lambda)=\sum_{k=1}^t n_k(p_k-1)=\sum_{k=1}^t (n-n_k).
\end{equation}
\end{remark}

\begin{remark}
Another way to view the braided Hopf algebra $H_n$ is as a braided tensor product
$$ H_n\cong u_{\xi_1^{n_1}}^+(\mathfrak{sl}_2)\otimes \ldots \otimes u_{\xi_t^{n_t}}^+(\mathfrak{sl}_2)$$
of positive parts $u_{\xi_k^{n_k}}^+(\mathfrak{sl}_2)\cong \Bbbk[\d_k]/(\d_k^{p_k})$ of the small quantum group at $p_k$th root of unity $\xi_k^{n_k}$. This follows using \cite{AS2}*{Lemma~4.2}.
\end{remark}

\subsection{The bosonization of \texorpdfstring{$H_n$}{Hn}}
In order to study modules over $H_n$ in terms of rational graded modules, we consider the bosonization $H_n\rtimes \Bbbk C_N$. Using Section \ref{ratgraded-sect}, $H_n$ is a Hopf algebra object in  $\lomod{C_N}_q$. Hence, we can form the bosonization $H_n\rtimes \Bbbk C_N$ \cite{Rad}.

\begin{lemma}\label{lem-bosonized-Hn}
The Hopf algebra $H_n\rtimes \Bbbk C_N$ is generated by the elements $\d_1,\ldots, \d_t$ and $K$ as a $\Bbbk$-algebra, subject to the algebra relations
\[
\begin{array}{cccc}
K^N=1, &&& K\d_k=\xi_k\d_k K, \\ 
\d_k^{p_k}=0, &&& [\d_k,\d_l]=0.
\end{array}
\]
The coproduct, antipode and counit  are given on the generators by
\[
\begin{array}{cccc}
\Delta(K)=K\otimes K, &&& \Delta(\d_k)=\d_k\otimes 1+K^{n_k}\otimes \d_k,\\
S(K)=K^{-1}, &&& S(\d_k)=-K^{-n_k}\d_k,\\
\epsilon(K)=1, &&&\epsilon(\d_k)=0.
\end{array}
\]
\end{lemma}
\begin{proof}
This follows using \cite{Maj1}*{Theorem 9.4.12}.
\end{proof}

Inductively, we obtain the formula
\begin{align}\label{coproduct-higher}
\Delta(\d_k^a)&=\sum_{i=0}^a {a \brack i}_{\xi_k^{n_k}}\d_k^i K^{(a-i)n_k}\otimes \d_k^{a-i},
\end{align}
for any integer $a\geq 0$. Using $K^{n_k}\d_l=\d_lK^{n_k}$ for $k\neq l$, we derive a more general formula. For this, given a $t$-tupel of non-negative integers $\mathbf{a}=(a_1,\ldots, a_t)\in\mN_0^t$, we write $\d^{\mathbf{a}}=\d_1^{a_1}\ldots\d_t^{a_t}$ and $K^{\mathbf{a}}=K^{a_1n_1}\ldots K^{a_tn_t}$. Then
\begin{align}\label{coproduct-higher2}
\Delta(\d^{\mathbf{a}})&=\sum_{\mathbf{b}} \left(\prod_{j=1}^t{a_j \brack b_j}_{\xi_j^{n_j}}\right) \d^{\mathbf{b}}K^{\mathbf{a}-\mathbf{b}}\otimes \d^{\mathbf{a}-\mathbf{b}},
\end{align}
where the sum is taken over all $\mathbf{b}=(b_1,\ldots, b_t)\in \mN_0^t$ such that $b_k\leq a_k$ for all $k$, and $\mathbf{a}-\mathbf{b}=(a_1-b_1,\ldots, a_t-b_t)\in \mN^t_0$.

\begin{lemma}\label{integrals}
The element $\Lambda'=\sum_i K^i\d_1^{p_1-1}\ldots\d_t^{p_t-1}$ is a left integral in $H_n\rtimes \Bbbk C_N$.
\end{lemma}
\begin{proof}
We have to show that $h\Lambda'=\epsilon(h)\Lambda'$ for all $h\in H_n\rtimes \Bbbk C_N$. It suffices to check the property on generators, on which it is evident.
\end{proof}

Note that the trace map $\mathrm{Tr}$ from Lemma \ref{lem-Hn-Frobenius-Hopf} is related to $\Lambda$ in the following way. First, there is a non-degenerate Hopf pairing  $\langle~,~\rangle\colon (H_n\rtimes \Bbbk C_N)\otimes (H_n\rtimes \Bbbk C_N)\to \Bbbk$ obtained by extending the pairing $\langle ~,~\rangle$ from Remark \ref{nichols} via
\begin{equation}\langle \d^{\mathbf{a}}\otimes K^{i},\d^{\mathbf{b}}\otimes K^{j}\rangle=\langle\d^{\mathbf{a}} ,\d^{\mathbf{b}} \rangle q^{ij}.\end{equation}
Thus, $H_n\rtimes \Bbbk C_N$ is self-dual as a Hopf algebra. Following \cite{LS}, we obtain another, so-called \emph{right orthogonal}, pairing $(~,~)$ on $(H_n\rtimes \Bbbk C_N)\otimes (H_n\rtimes \Bbbk C_N)$ by the formula
$$ (\d^{\mathbf{a}}\otimes K^{i},\d^{\mathbf{b}}\otimes K^{j})=\sum_{\Lambda} \langle \d^{\mathbf{a}}\otimes K^{i},\Lambda_1\rangle\langle \d^{\mathbf{b}}\otimes K^{i},\Lambda_2\rangle=\langle \d^{\mathbf{a}-\mathbf{b}}\otimes K^{i-j},\Lambda\rangle.$$
Restricting $(~,~)$ to $H_n\otimes H_n$ gives the pairing given by $\mathrm{Tr}(\d^{\mathbf{a}}\cdot \d^{\mathbf{b}})$ which makes $H_n$ a Frobenius algebra.

\subsection{A spherical structure}
In this section we show that the category $\lmod{H_n}$, viewed as a monoidal category using the tensor product structure of $\lratmod{H_n\rtimes \Bbbk C_N}$, is a spherical monoidal category (cf. \cite{BW}*{Section 2} or \cite{EGNO}*{Section 4.7}). 

\begin{lemma}\label{sphericalhopf}
The element  $\omega=K^{-\sum_{k=1}^t n_k}$ satisfies the following properties: 
\begin{itemize}
\item[(i)] It is group like in the sense that
\begin{equation}\label{spherical1}
\Delta(\omega)=\omega\otimes \omega, \qquad S(\omega)=\omega^{-1}, \qquad \epsilon(\omega)=1.
\end{equation}
\item[(ii)] Conjugating by $\omega$ implements $S^2$. That is, for any $h\in H_n\rtimes \Bbbk C_N$,
\begin{equation}\label{spherical2}
S^2(h)=\omega h\omega^{-1}.
\end{equation} 
\end{itemize} 
\end{lemma}
\begin{proof}
This follows from a simple computation using Lemma \ref{lem-bosonized-Hn}.
\end{proof}

The lemma shows that $H_n\rtimes \Bbbk C_N$ is \emph{almost} a spherical Hopf algebra, with only condition (5) of \cite{BW}*{Definition~3.1} missing. However, working with the full subcategory $\lmod{H_n}$, this condition will always hold to give the following result:

\begin{proposition}
The monoidal category $\lmod{H_n}$ is a spherical category.
\end{proposition}
\begin{proof}
This follows using \cite{BW}*{Theorem~3.6}. In fact, the conditions from Lemma \ref{sphericalhopf} give that $\lomod{H_n\rtimes \Bbbk C_N}$ is a pivotal category \cite{BW}*{Definition 2.1}. We observe that for $V$ a rational graded $H_n$-module, $\omega$ acts by 
\begin{align*}
\omega\cdot v&=q^{-i(\sum_{k=1}^t n_k)}v, &&\forall v\in V^i.
\end{align*}
Thus, for any graded morphism $\theta\colon V\to V$ of $H_n$-modules, $\omega \theta=\theta\omega$. This shows that $\lratmod{H_n\rtimes \Bbbk C_N}$ is a spherical category.
\end{proof}

\subsection{A weak replacement for the braiding}

In general, the category $\lmod{H_n\rtimes \Bbbk C_N}$ and its subcategory $\lmod{H_n}$ are not braided monoidal. This agrees with the observation of \cite{Bic}*{Proposition 5} that the category of $n$-complexes is not braided monoidal (unless $q=q^{-1}$). 

A further observation is that, as an algebra, $H_n$ does not depend on the parameter $q$. The coproduct and $H_n$-module structure, however, are dependent on $q$, manifested in the use of the braiding in $\qvect$. For any choice of a primitive $N$th root of unity, we have two different coproducts on $H_n\rtimes \Bbbk C_N$  --- the coproduct $\Delta=\Delta_q$ from Lemma \ref{lem-bosonized-Hn}, and its opposite coproduct $\Delta^{\oop}=\Delta_q^{\oop}$. The tensor product obtained from the former is denoted by $\otimes=\otimes_q$ for the purpose of this section. Note the symmetry that
\begin{align*}
\Delta_q(\d_k)&=\d_k\otimes 1+K^{n_k}\otimes \d_k, \\
\Delta_{q^{-1}}(\d_k)&=\d_k\otimes 1+K^{-n_k}\otimes \d_k,
\end{align*}
are distinct coproducts for bosonizations of $H_n$, utilizing $\otimes_q$ or $\otimes_{q^{-1}}$, respectively.

In this section, we describe a weaker symmetry that is present in place of a quasi-triangular structure on $H_n\rtimes \Bbbk C_N$. A quasi-triangular structure would give natural isomorphisms $V\otimes W\cong W\otimes V.$
Instead, we obtain the following. 

\begin{proposition}\label{weakbraiding}
There are natural isomorphisms of graded $H_n$-modules
$$\Psi_{V,W}\colon V\otimes_q W\cong W\otimes_{q^{-1}} V, \quad \quad \Psi_{V,W}(v\otimes w)= q^{-ij}w\otimes v$$
where $v\in V^i$, $w\in W^j$ are homogeneous elements.
\end{proposition}
\begin{proof}
The proposition can be checked by a direct computation that $\Psi_{V,W}$ intertwines with the action of the $H_n$ generators $\d_k$, $k=1,\dots, t$. More intrinsically, consider the universal R-matrix $R$ for $\Bbbk C_N$ from equation \eqref{Rmatrixgroup}. Now, $R$ is a right $2$-cycle for $\Bbbk C_N$, and also for $H_n\rtimes \Bbbk C_N$ which contains  $\Bbbk C_N$ as a Hopf subalgebra. Hence, we can consider the \emph{Drinfeld twist} $\Delta^R_q=R^{-\ast} \Delta_q R$ of the coproduct of $H_n\rtimes\Bbbk C_N$ \cite{Dri2}. We compute that
\begin{align*}
\Delta^R_q(\d_k)&=\sum_{i,j,a,b}q^{ij-ab}K^i \d_k K^a\otimes K^j K^b+ \sum_{i,j,a,b}q^{ij-ab}K^i K^{n_k}K^a\otimes K^j \d_kK^b\\
&=\sum_{i,j,a,b}q^{i(j+n_k)-ab}\d_k K^{i+a}\otimes K^{j+b}+ \sum_{i,j,a,b}q^{(i+n_k)j-ab}K^{i+n_k+a}\otimes \d_k K^{j+b}\\
&=\sum_{i,j,a,b}\left(q^{ij-ab}\d_k K^{i+a}\otimes K^{j-n_k+b}+ q^{ij-ab}K^{i+a}\otimes \d_k K^{j+b}\right)\\
&=\left(\d_k\otimes K^{-n_k}+1\otimes \d_k\right)\sum_{i,j,a,b}q^{ij-ab}K^{i+a}\otimes K^{j+b}\\&=\left(\d_k\otimes K^{-n_k}+1\otimes \d_k\right)=\Delta^{\oop}_{q^{-1}}(\d_k).
\end{align*}
The result follows.
\end{proof}


\section{A tensor ideal in \texorpdfstring{$\lmod{H_n}$}{Hn-mod}}\label{sec-tensorideal}

\subsection{The category of \texorpdfstring{$H_n$}{Hn}-modules}\label{Hnmodules}

We use the same notation as in the previous sections, and work with the category $\lmod{H_n}$ of finite-dimensional graded $H_n$-modules. This category has internal homs $\Hom^\bullet(V,W)\cong V^*\otimes W $. The differential $\d_k\in H_n$ acts on an element $f\colon V\to W$, for $v$ homogeneous of degree $i$, by 
\begin{align}\label{homaction}
(\d_k\cdot f)(v)= \xi_k^{-i}(\d_kf(v)-f(\d_kv)),
\end{align}
as in equation \eqref{internalhom}\footnote{Note that this formula differs slightly from \cite{Kap}*{equation (1.14)}. A formula similar to that of Kapranov is obtained by using the alternative internal hom from Remark \ref{alternative-int-hom}. In this case, we would obtain $(\d_kf)(v)=d_kf(v)-\xi_k^{\deg(f)}f(\d_k v)$. The results of this section apply using either convention.}, where $\xi_k=q^{n_k}$. Hence $\d_k \cdot f=0$ if and only if $f(\d_k v)=\d_k f(v)$ for all $v\in V$. In particular, a linear map is graded $H_n$-invariant if and only if it is of degree zero and commutes with all differentials.  In this way, the category of $H_n$-modules is enriched over itself.

As $H_n$ is naturally a $\mZ$-graded algebra, we have the grading shift functors on $\lmod{H_n}$
$$\shift{k}\colon \lmod{H_n}\longrightarrow \lmod{H_n}$$
for all $k\in \mZ$.
Equivalently, consider the modules $\Bbbk\shift{\pm 1}$, which are one-dimensional over $\Bbbk$, with generators $\one$ sitting respectively in $\mZ$-degrees $\mp 1$. Then $V\shift{\pm 1}\cong V\otimes \Bbbk\shift{ \pm 1}$. Indeed, for all $v_i\in V_i$,
\begin{align*}
\d_k  (v_i\otimes \one)&=(\d_kv_i)\otimes \one.
\end{align*}
This shows that the isomorphism $V\otimes \Bbbk\shift{\pm 1}\to V\shift{\pm 1}$ sending $v_i\otimes \one$ to $v_i$ commutes with the $\d_k$-action, for $v_i\otimes \one$ has degree $i\mp 1$.

\begin{lemma}\label{swapiso}
For any two $H_n$-modules $V, W$, there are natural isomorphisms of $H_n$-modules 
$$V\otimes (W\shift{\pm 1})\cong (V\otimes W)\shift{\pm 1}\cong (V\shift{ \pm 1})\otimes W. $$
\end{lemma}
\begin{proof}
We show the $\shift{1}$ case. Since $W\shift{1}\cong W\otimes \Bbbk\shift{1}$, the first isomorphism is easy. To establish the second isomorphism, we consider the isomorphisms of $H_n$-modules
\[
V\shift{1}\otimes W \cong  (V\otimes \Bbbk\shift{1})\otimes W \cong  V\otimes (\Bbbk\shift{1}\otimes W ),
\]
which reduces the problem to showing that $\Bbbk\shift{1}\otimes W\cong W\otimes \Bbbk\shift{1}$.

Denote by $\one$ a generator of $\Bbbk\shift{1}$ which lives in degree $-1$. We define the map $r_V$, for any homogeneous $v_i \in V_i$,  by
\begin{align*}
r_V(v_i\otimes \one):= q^{-i} \one \otimes v_i.
\end{align*}
It follows that, for any $k=1,\ldots, t$,
\begin{align*}
\d_k(r_V(v_i\otimes \one))&=q^{-i}\d_k(\one \otimes v_i)= q^{-i}\xi_k^{-1}\one \otimes \d_k v_i\\&=q^{-i-n_k}\one \otimes \d_kv_i=r_V(\d_kv_i\otimes \one )=r_V(\d_k(v_i\otimes \one)),
\end{align*}
proving that $r_V$ is a morphism of $H_n$-modules. Naturality is clear as any morphism $f\colon V\to W$ of $H_n$-modules preserves the grading, and hence
\begin{align*}
r_W(f(v_i)\otimes \one)=q^{-i}(\one \otimes f(v_i))=\one \otimes f(q^{-i}v_i)=(\ide\otimes f)r_V(v_i\otimes \one).
\end{align*}
The grading shift $\shift{-1}$ is similar, and one just replaces $q$ by $q^{-1}$ in the above computations.
\end{proof}

\begin{corollary}\label{gradingshift}
Let $V$, $W$ be $H_n$-modules. For any $k\in \mZ$, there are isomorphisms of $H_n$-modules
\begin{gather}
(V\shift{k})\otimes W\cong (V\otimes W)\shift{k}\cong V\otimes (W\shift{k}),\label{gradingiso1}\\
\Hom^\bullet(V\shift{-k},W)\cong \Hom^\bullet(V,W\shift{k})\cong \Hom^\bullet(V,W)\shift{k}.\label{gradingiso2}
\end{gather}
\end{corollary}
\begin{proof}
The first equation \eqref{gradingiso1} is a repeated application of the previous Lemma \ref{swapiso}.

Using equation \eqref{internalhom_dual} and the first part of the corollary, we have the chain of isomorphisms of $H_n$-modules
\begin{align*}
\Hom^\bullet(V\shift{-k},W)&\cong(V\shift{-k})^* \otimes W \cong (\Bbbk\shift{-k}\otimes V)^*\otimes W \\
&\cong ( V^*\otimes\Bbbk\shift{-k}^*)\otimes W \cong V^*\otimes( \Bbbk\shift{k} \otimes W) \\&\cong  V^*\otimes W\shift{k}\cong \Hom^\bullet(V,W\shift{k}).
\end{align*}
The last isomorphism in the second equality \eqref{gradingiso2} is established in a similar way. 
\end{proof}

As a special case, we can consider $n=p^a$. 
In this case, we can fully classify indecomposable modules over $H_{n}$. Any indecomposable $H_n$-module is isomorphic to a grading shift of a quotient module $H_n/(\d_1^l)$, for $l=0,\ldots, p_1-1$. Such a simple classification is not possible in the presence of more than two distinct prime factors in $n$.

\subsection{The tensor ideals \texorpdfstring{$\bfI_k$}{Ik}}
Once again, fix a positive integer $n$ and its prime decomposition $n=p_1^{a_1}\cdots p_t^{a_t}$, and let us consider the category $\lmod{H_n}$ of finite-dimensional graded modules over the braided Hopf algebra $H_n$.

The braided Hopf algebra $H_n$ has many useful Hopf subalgebras. For each prime factor $p_k$, let us consider two complementary Hopf subalgebras inside $H_n$:
\begin{equation}
H_n^k:=\dfrac{\Bbbk[\d_k]}{(\d_k^{p_k})},\quad \quad \quad
\widehat{H}_n^k:=\dfrac{\Bbbk[\d_1,\dots, \widehat{\d_k},\dots \d_t]}{(\d_1^{p_1},\dots, \widehat{\d_k^{p_k}},\dots \d_t^{p_t})}.
\end{equation}
Here the ``hatted'' terms in the second equation are dropped from the expressions. Each $\d_k$ has degree $n_k:=n/p_k$. If $t=1$, i.e., $n=p_1^{a_1}$ is a prime power, we shall not consider $\widehat{H}_n^1$, and $H_n^1=H_n$.

We record the following simple observation.

\begin{lemma}\label{indproj}
The left regular module is, up to isomorphism and grading shift, the only indecomposable projective-injective $H_n$-module. Its graded dimension equals
\[
\mathrm{dim}_\nu(H_n)=\prod_{k=1}^t(1+\nu^{n_k}+\cdots+\nu^{(p_k-1)n_k})=\prod_{k=1}^t\dfrac{1-\nu^n}{1-\nu^{n_k}}.
\]
\end{lemma}
\begin{proof}
This follows since $H_n$ is a graded Frobenius local algebra (Lemma \ref{lem-Hn-Frobenius-Hopf}), and thus is graded self-injective. The graded dimension computation is an easy exercise.
\end{proof}

\begin{definition}\label{bfIk-def}
For each prime factor $p_k$ of $n$, we define a $p_k$-dimensional graded $H_n$-module $V_k$ by
\[
V_k:=\mathrm{Ind}_{\widehat{H}_n^k}^{H_n}(\Bbbk)\cong H_n \otimes_{\widehat{H}_n^k}\Bbbk.
\]

Further, if $t>1$, we denote 
\[
W_k:=\mathrm{Ind}_{H_n^k}^{H_n}(\Bbbk)\cong H_n \otimes_{H_n^k}\Bbbk,
\]
which is an $m_k$-dimensional $H_n$-module.
\end{definition}

Observe that the $H_n$-module $V_{k}$ is a $p_k$-fold extension of the trivial $H_n$-module $\Bbbk$ by itself:
\begin{equation*}
\xymatrix{
\Bbbk\ar[r]^-{\d_k}&\Bbbk\shift{-n_k}\ar[r]^-{\d_k}&\cdots\ar[r]^-{\d_k}&\Bbbk\shift{(2-p_k)n_k}\ar[r]^-{\d_k}&\Bbbk\shift{(1-p_k)n_k}.}
\end{equation*}
We further observe that $V_k$ is isomorphic as an $H_n$-module, up to grading shift, to the submodule of $H_n$ generated by the element $\d_{1}^{p_{1}-1}\cdots \widehat{\d_{k}^{p_k-1}}\cdots\d_{t}^{p_{t}-1}$. Similarly, $W_k$ is isomorphic to a grading shift of the submodule generated by $\d_k^{p_k-1}$. It follows similarly to Lemma \ref{indproj} that 
\begin{align}
\mathrm{dim}_\nu(V_k)&=\dfrac{1-\nu^n}{1-\nu^{n_k}}, &\mathrm{dim}_\nu(W_k)&=\prod_{l\neq k}\dfrac{1-\nu^n}{1-\nu^{n_l}}.
\end{align}
The module $V_k$ is free when viewed as an $H_n^k$-module, while $W_k$ is free as an $\widehat{H}_n^k$ module. In particular, $W_k$ is free as an $H_n^l$-module for all $l\neq k$.


\begin{definition}\label{def-Ik}
 Assume that $t>1$. For any $k=1,\ldots, t$, we let $\bfI_k$ be the full subcategory of modules in $\lmod{H_n}$ consisting of direct summands of $H_n$-modules $V$ of the following form:
\begin{enumerate}
\item[(i)] $V$ is equipped with a finite-step filtration by $H_n$-submodules: $0=F_{0}\subset F_1\subset F_2\subset \dots \subset F_r=V$.
\item[(ii)] Each of the subquotient modules $F_i/F_{i-1}$ ($i=1,\dots, r$) is isomorphic to $W_k$ up to a grading shift.
\end{enumerate}
If $t=1$, so that $n=p_1^{a_1}$, we denote $\bfI_1:=\bfI_{H_n}$, the full subcategory of graded projective-injective $H_n$-modules, cf. Section \ref{stable-sect}.
\end{definition}

\begin{lemma}\label{lem-ext-closed-Ik}
The ideal $\bfI_k$ is closed under extensions. More precisely, if $U$, $V$ and $W$ fit into a short exact sequence of $H_n$-modules
\[
0\longrightarrow U\longrightarrow W \stackrel{\pi}{\longrightarrow} V\longrightarrow 0
\]
with $U,V$ being in $ \bfI_k$, then $W$ also lies in $\bfI_k$.
\end{lemma}
\begin{proof} The case when $t=1$ is clear, so we assume $t>1$.
Assume given such a short exact sequence of $H_n$-modules such that $U'$, $V'$ be $H_n$-modules satisfying that $U\oplus U'$ and $V\oplus V'$ are equipped with filtrations $F_1\subset \dots \subset F_r$ and $F^\prime_1\subset \dots \subset F^\prime_s$ as in Definition \ref{def-Ik}. Then we have a short exact sequence
$$0\to U\oplus U'\to W\oplus U'\oplus V'\to V\oplus V'\to 0,$$
and  $W\oplus U'\oplus V'$ is equipped with a filtration
\[
0\subset F_1\subset \dots \subset F_r=U \subset \pi^{-1}(F^\prime_1)\subset \dots \subset \pi^{-1}(F^\prime_s)=W,
\]
which satisfies the hypothesis of Definition \ref{def-Ik}. Hence $W$, as a direct summand of $W\oplus U'\oplus V'$, is contained in $\bfI_k$.
\end{proof}

\begin{lemma}\label{lem-action-closed-Ik}
The ideal $\bfI_k$ is closed under forming duals and tensor products with arbitrary objects in $\lmod{H_n}$. Consequently, $\bfI_k$ is a two-sided tensor ideal in $\lmod{H_n}$.
\end{lemma}
\begin{proof} The case $t=1$ follows from \cite{Kh}*{Proposition 2}. Hence, we assume $t>1$.
If $V$ is a direct summand of an object $W$ of $\bfI_k$ with a filtration $F_\bullet$, then $W^*$ is equipped with the dual filtration $F_\bullet^*$, which is readily checked to satisfy the conditions of Definition \ref{def-Ik}. Hence $V^*$, as a direct summand of $W^*$, is an object in $\bfI_k$.

Suppose $V\in \bfI_k$ and $U$ is any $H_n$-module. The module $U$ has a nontrivial socle since $H_n$ is a graded local algebra. Choose $\Bbbk\shift{s}$ lying inside the socle of $U$, which gives us a short exact sequence of $H_n$-modules
\[
0\longrightarrow \Bbbk\shift{s} \longrightarrow U  \longrightarrow \bar{U} \longrightarrow 0.
\]
Tensoring, for instance, on the left with $V$, we obtain
\[
0\longrightarrow V\shift{s} \longrightarrow V\otimes U  \longrightarrow V\otimes \bar{U} \longrightarrow 0.
\]
By induction on $\mathrm{dim}(U)$, we may assume that $V\otimes \bar{U}\in \bfI_k$ (the case $\mathrm{dim}(U)=1$ is the assumption that $V\in \bfI_k$). Now the previous lemma applies and shows that $V\otimes U\in \bfI_k$.
\end{proof}

It follows that the internal homs also preserve the ideals $\bfI_k$.

\begin{corollary}\label{cor-hom-closed-Ik}
Let $U$ be an $H_n$-module in the ideal $\bfI_k$ and $V$ be an arbitrary finite-dimensional $H_n$-module. Then both $\Hom^\bullet(U,V)$ and $\Hom^\bullet(V,U)$ are objects of $\bfI_k$.
\end{corollary}
\begin{proof}
This follows from Lemma \ref{lem-action-closed-Ik} and the isomorphism of graded $H_n$-modules $\Hom^\bullet(U,V)\cong  U^*\otimes V$ from equation \eqref{internalhom_dual}.
\end{proof}

\begin{remark}
We note that the category $\bfI_k$ is the smallest subcategory of $\lmod{H_n}$ closed under grading shifts, extensions, and direct summands that contains the objects $W_k$. 
We conjecture that any object in $\bfI_k$ in fact has a filtration as in Definition \ref{def-Ik}.
\end{remark}

\begin{lemma}\label{lem-proj-inj-closed-Ik}
The class of projective-injective objects of $\lmod{H_n}$ is contained in each $\bfI_k$, for $k=1,\dots, t$.
\end{lemma}
\begin{proof}
 This follows since we have 
\begin{equation}
H_n=\mathrm{Ind}_{H_n^k}^{H_n}H_n^k,
\end{equation}
and the  regular $H_n^k$-module is an iterated extension of grading shifts of the trivial $H_n^k$-module.
\end{proof}

\begin{example}\label{example1}
Let $n=2^a\cdot 3^b$, with $a,b\geq 1$. Then $\d_1$ raises degrees by $n_1=2^{a-1}3^b$, and  $\d_2$ raises degrees by $n_2=2^a3^{b-1}$. We note that $V_k=W_k$ in the case of only two distinct prime factors.
Let us consider the following module $V$ with the non-zero differential acting by identity maps indicated on the arrows:
\begin{align*}
V: \quad \quad
\begin{gathered}
\xymatrix{
&\Bbbk\ar[d]^{\d_1}\ar[r]^-{\d_2}&\Bbbk\shift{-n_2}\ar[r]^-{\d_2}\ar[d]^{\d_1}&\Bbbk\shift{-2n_2}\\
\Bbbk\shift{n_2-n_1}\ar[r]^-{\d_2}&\Bbbk\shift{-n_1}\ar[r]^-{\d_2}&\Bbbk\shift{-n_1-n_2}&
}
\end{gathered} \ .
\end{align*} 
The module $V$ is contained in the ideal $\bfI_2$ (note that $p_2=3$ here).  Note that $V$ does not split as a direct sum of shifts of $W_{2}$, but we see that there is a short exact sequence of $H_n$-modules
$$0\longrightarrow W_{2}\shift{n_2-n_1}\longrightarrow V\longrightarrow W_{2}\longrightarrow 0.$$

If $n=2^a\cdot 3^b\cdot 5^c$, there exist various non-split extensions in $\bfI_1$. For example, consider the module $W$, where we omit the degree shifts,
\begin{align*}
W: \quad \quad
\begin{gathered}
\xymatrix@!0{
&&&&& \Bbbk  \ar[rr]^-{\d_3}& & \Bbbk \ar[rr]^-{\d_3}& & \Bbbk \ar[rr]^-{\d_3}& & \Bbbk  \ar[rr]^-{\d_3}&&\Bbbk\\
&&&& \Bbbk  \ar[ur]^-{\d_2}\ar[rr]^-{\d_3}\ar'[d][dd]^-{\d_1}& & \Bbbk \ar[ur]^-{\d_2}\ar[rr]^-{\d_3}\ar'[d][dd]^-{\d_1}& & \Bbbk\ar[ur]^-{\d_2} \ar[rr]^-{\d_3}\ar'[d][dd]^-{\d_1}& & \Bbbk \ar[ur]^-{\d_2} \ar[rr]^-{\d_3}\ar'[d][dd]^-{\d_1}&&\Bbbk\ar[ur]^-{\d_2}\\
&&&\Bbbk\ar[ur]^-{\d_2}\ar[rr]\ar[dd]& & \Bbbk\ar[ur] \ar[dd]\ar[rr]& & \Bbbk\ar[ur] \ar[dd]\ar[rr]& & \Bbbk\ar[dd]\ar[ur]\ar[rr]&&\Bbbk\ar[ur]_-{\d_2}\\
&&\Bbbk\ar'[r][rr] && \Bbbk\ar'[r][rr]  && \Bbbk\ar'[r][rr] && \Bbbk\ar'[r][rr]  & & \Bbbk  &&\\
&\Bbbk\ar[ur]^-{\d_2}\ar[rr]^-{\d_3} &&\Bbbk\ar[ur]\ar[rr]^-{\d_3}&&\Bbbk \ar[ur]\ar[rr]^-{\d_3}&&\Bbbk\ar[ur]\ar[rr]^-{\d_3}& & \Bbbk\ar[ur]_-{\d_2} &&&\\
\Bbbk\ar[ur]^-{\d_2}\ar[rr]^-{\d_3} &&\Bbbk\ar[ur]\ar[rr]^-{\d_3}&&\Bbbk \ar[ur]\ar[rr]^-{\d_3}&&\Bbbk\ar[ur]\ar[rr]^-{\d_3}& & \Bbbk\ar[ur]_-{\d_2} &&&&
}
\end{gathered}
\end{align*}
The $H_n$-module $W$ is free over $H_n^1$ and fits into a non-split short exact sequence
$$0\longrightarrow W_{2}\shift{-n_1+n_2+n_3}\longrightarrow W\longrightarrow W_{2}\longrightarrow 0.$$
\end{example}


\subsection{The tensor ideal \texorpdfstring{$\bfI$}{I}}

In order to capture rings of cyclotomic integers via categorification, we shall work with a larger ideal $\bfI$ in $\lmod{H_n}$ than that of projective-injective objects and containing each $\bfI_k$. This can be thought of as a type of ``sum'' of the ideals $\bfI_k$.

\begin{definition}\label{defnI}
Let $\bfI$ be the full subcategory of $\lmod{H_n}$ which consists of objects $U=\bigoplus_{k=1}^t U_k$, where $U_k$ is an object in $\bfI_k$. 
\end{definition}

\begin{lemma}\label{lem-ext-closed}
The ideal $\bfI$ is closed under grading shifts, forming duals, and taking tensor products with arbitrary objects of $\lmod{H_n}$. Consequently, $\bfI$ is a two-sided tensor ideal in $\lmod{H_n}$.
\end{lemma}
\begin{proof}
This is a consequence of Lemma \ref{lem-action-closed-Ik}.
\end{proof}

\begin{corollary}
Let $U$ be an $H_n$-module in the ideal $\bfI$ and $V$ be an arbitrary finite-dimensional $H_n$-module. Then both $\Hom^\bullet(U,V)$ and $\Hom^\bullet(V,U)$ are objects of $\bfI$.
\end{corollary}
\begin{proof}
This follows from Lemma \ref{lem-ext-closed} and the isomorphism $\Hom^\bullet(U,V)\cong U^*\otimes V $ of $H_n$-modules from equation \eqref{internalhom_dual}.
\end{proof}

\begin{lemma}\label{lem-summands-closed}
The ideal $\bfI$ is closed under taking direct summands. That is, if $W\in \bfI$ and $W\cong U\oplus V$, then both $U$ and $V$ belong to $\bfI$. 
\end{lemma} 
\begin{proof}
This follows from the fact that $\lmod{H_n}$ has the Krull-Schmidt property.
\end{proof}

The ideal $\bfI$ is not closed under extensions. However, its image in the stable category $\slmod{H_n}$ will possess the two-out-of-three property (see Lemma \ref{2of3-lemma}) based on the following proposition which generalizes \cite{Mir}*{Theorem 3.5} in our setup.

\begin{proposition}\label{ortho-prop}
Let $p_k$ and $p_l$ be distinct prime factors of $n$.
Let $V$ be an object in $\bfI_k$ and $W$ an object in $\bfI_l$. Then $\Hom^\bullet_{\slmod{H_n}}(V,W)\subset \bfI^\bullet_{H_n}(V,W)$. That is, all $H_n$-morphisms from $V$ to $W$ are null-homotopic.
\end{proposition}
\begin{proof}
We first show the statement for $V=W_k$ and $W=W_l$, $k\neq l$. According to Theorem~\ref{thm-hom-as-stable-invariant},
$$\Hom^\bullet_{\slmod{H_n}}(W_k,W_l)=\frac{(W_k^*\otimes W_l)^{H_n}}{\Lambda\cdot (W_k^*\otimes W_l)}.$$
Since $k\neq l$, we can equip $W_k$ with a filtration of $H_n$-modules whose successive quotients are grading shifts of the module $V_l$. Hence $W_k^*$ also has such a filtration. Next, we observe that 
the tensor product $V_l\shift{s}\otimes W_l$ is free over $H_n$. Inductively, it follows from the exactness of $\otimes$ that $W_k^*\otimes W_l$ has a (split) resolution by free $H_n$-modules and is hence free. Therefore, $\Lambda\cdot (W_k^*\otimes W_l)=(W_k^*\otimes W_l)^{H_n}$ and we have shown that $\Hom^\bullet_{\slmod{H_n}}(W_k,W_l)=\{0\}$.

Using Corollary \ref{gradingshift}, we can replace $W_k$, $W_l$ by grading shifts. Thus, the statement holds for all modules $V$ in $\bfI_k$ and $W$ in $\bfI_l$ that have filtrations as in Definition \ref{def-Ik}. If $U_V$ is a direct summand of $V$ and $U_W$ a direct summand of $W$ and $f\colon U_V\to U_W$ an $H_n$-module morphism. Then $f$ extends by zeros to a $H_n$-morphism $V\to W$, which is null homotopic by the above. Hence, $f$ is also null-homotopic, and the statement is proved for general objects in $\bfI_k$ and $\bfI_l$.
\end{proof}

%
%
%

%

\section{Categorifying cyclotomic rings}\label{sec-catfn}
In this section, we construct a tensor triangulated category $\mathcal{O}_n$, whose Grothendieck ring is isomorphic to the cyclotomic ring $\mathbb{O}_n$ at an $n$th root of unity.

\subsection{A triangulated quotient category}
Consider the stable category $\slmod{H_n}$ from Section \ref{stablesect} which is tensor triangulated. Let us denote  by $\un{\bfI}$ the full subcategory consisting of objects that are isomorphic to those of $\bfI$ under the natural quotient functor $\lmod{H_n}\longrightarrow \slmod{H_n}$. Thus $\underline{\bfI}$ is a strictly full subcategory of $\slmod{H_n}$. Our first goal is to show that $\un{\bfI}$ is a thick triangulated subcategory in $\slmod{H_n}$. To do this we first exhibit some preparatory results.

\begin{lemma}\label{extensionlemma}
The subcategory $\underline{\bfI}$ is closed under the tensor product action by $\slmod{H_n}$. More precisely, if $U$ is an object of $\underline{\bfI}$ and $V\in \slmod{H_n}$, then both $V\otimes U$ and $U\otimes V$ are in $\underline{\bfI}$. Consequently, $\underline{\bfI}$ constitutes a tensor ideal in $\slmod{H_n}$.
\end{lemma}
\begin{proof}
We may take $U$ to be the image of an object of $\bfI$ under the quotient functor. The lemma is then a consequence of Lemma \ref{lem-ext-closed} of the previous subsection.
\end{proof}

\begin{corollary}\label{cor-shift-closed}
The subcategory $\underline{\bfI}$ is closed under the homological shifts of $\slmod{H_n}$.
\end{corollary}
\begin{proof}
This follows from the previous Lemma and the fact that
\[
U[1]\cong U\otimes (H_n/\Bbbk\Lambda)\shift{\l}.
\]
for any object $U\in\slmod{H_n}$.
\end{proof}

\begin{lemma}\label{2of3-lemma}
Let $U\longrightarrow V\longrightarrow W\longrightarrow U[1]$ be a distinguished triangle in $\slmod{H_n}$. If two out of the three objects $U$, $V$ and $W$ are in $\underline{\bfI}$, then so is the third object.
\end{lemma}

\begin{proof}
Using Corollary \ref{cor-shift-closed} and the fact that any distinguished triangle is isomorphic to a standard distinguished triangle, we are reduced to showing that, if $U$, $V$ are objects of $\underline{\bfI}$ and $f\colon U\longrightarrow V$ is a map of $H_n$-modules, then the cone $C_f$ of $f$ is also in $\bfI$.

There exist direct sum decompositions $U\cong \oplus_{k=1}^t X_k$ and $V\cong\oplus_{l=1}^t Y_l$, with $X_k,Y_k\in \bfI_k$. Under these isomorphisms, $f=(f_{kl})$ is a matrix of $H_n$-module maps, where $f_{kl}=\pi_{Y_l}f\iota_{X_k}$ for the canonical inclusion $\iota_{X_k}\colon X_k\to U$ and projection $\pi_{Y_k}\colon V\to Y_k$. It follows from Proposition \ref{ortho-prop} that the images of the components $f_{kl}$ are zero in $\slmod{H_n}$. Hence, we may replace $f$ by the diagonal $H_n$-module map $f'=(\delta_{k,l}f_{kk})$ which has an isomorphic cone in $\slmod{H_n}$. Further, the cone construction respects direct sums of morphisms, i.e. $C_{f\oplus g}\cong C_f\oplus C_g$. Hence, it suffices to show that $C_{g}$ is in $\bfI$ for any morphism $g\colon U'\to V'$, where $U',V'$ are objects in $\bfI_k$.

By the definition of distinguished triangles, see equation \eqref{eqn-triangle}, the cone $C_{g}$ fits into the diagram
\[
\begin{gathered}
\xymatrix{
0 \ar[r] & U \ar[d]_g\ar[r]^-{\rho_U} & U\otimes H_n \ar[d]\ar[r] & U[1]
\ar@{=}[d]\ar[r] &0\\
0 \ar[r] & V \ar[r] & C_g \ar[r]^h & U[1]\ar[r] & 0
}
\end{gathered}.
\]
By Lemma \ref{lem-ext-closed-Ik}, $C_g$ is an object in $\bfI_k\subseteq \bfI$.
\end{proof}

\begin{lemma}\label{lemma-thickness}
The ideal $\underline{\bfI}$ is closed under direct summands.  That is, if $W\cong U\oplus V$ as objects of $\underline{\bfI}$, then both $U$ and $V$ belong to $\underline{\bfI}$.
\end{lemma}
\begin{proof}
By adding enough projective-injective $H_n$-modules to both sides, we may assume that $W\cong U\oplus V$ in $\lmod{H_n}$. Thus the claim is a direct consequence of Lemma \ref{lem-summands-closed}.
\end{proof}

Recall that a full triangulated subcategory in a triangulated category is called \emph{thick} (or \emph{saturated}) if it is closed under taking direct summands (see e.g. \cite{St}*{\href{http://stacks.math.columbia.edu/tag/0123}{Tag 05RA}}). Lemma \ref{lemma-thickness} thus establishes the thickness of the ideal $\underline{\bfI}$ inside $\slmod{H_n}$.

Summarizing the above discussion, we have established the following.

\begin{theorem}
The ideal $\un{\bfI}$ constitutes a full triangulated tensor ideal in the stable category $\slmod{H_n}$ which is thick. \hfill $\square$
\end{theorem}

Hence standard machinery on localization allows us form a Verdier localized (or quotient) category of $\slmod{H_n}$ by $\underline{\bfI}$, see e.g. \cite{St}*{\href{http://stacks.math.columbia.edu/tag/0123}{Tag 05RA}}.

\begin{definition}
For any positive integer $n$, the category $\mathcal{O}_n$ is defined as the Verdier localization of $\slmod{H_n}$ by the ideal $\underline{\bfI}$:
$$\mathcal{O}_n:=\slmod{H_n}/{\un{\bfI}}.$$
\end{definition}

A morphism $s:M\rightarrow N$ in $\slmod{H_n}$ descends to an isomorphism in $\mathcal{O}_n$ if and only if the cone of $s$ is isomorphic to an object of $\underline{\bfI}$. We declare this class of morphisms $s$ as \emph{quasi-isomorphisms}. Such quasi-isomorphisms constitute a localizing class in $\slmod{H_n}$ since $\underline{\bfI}$ is a saturated full-subcategory of $\slmod{H_n}$. A general morphism from $M$ to $N$ in the localized category $\mathcal{O}_n$ is represented by a ``roof'' of the form
\begin{equation}
\begin{gathered}
\xymatrix{
& M^\prime \ar[dl]_s\ar[dr]^f&\\
M & & N
}
\end{gathered} \ ,
\end{equation}
where $s$ is a quasi-isomorphism and $f$ is some morphism in $\slmod{H_n}$.

\begin{remark}\label{rem-localize-vs-quotient}
The localization construction used is a also known as the Verdier quotient, cf. \cite{St}*{\href{http://stacks.math.columbia.edu/tag/0123}{Tag 05RA}}. Observe that a morphism $f\colon X\rightarrow Y$ in $\slmod{H_n}$ descends to zero in $\mathcal{O}_n$ if and only if it factors through an object of $\underline{\bfI}$. Indeed, the ``if'' part is clear, since any object of $\underline{\bfI}$ is isomorphic to the zero object in $\mathcal{O}_n$. Conversely, choose an $s:Y\rightarrow Y^\prime$ in $\slmod{H_n}$ such that $s\circ f=0$ and $s$ descends to an isomorphism in $\mathcal{O}_n$. Then the cone of $s$ shifted by $[-1]$, denoted by $C$, fits into the diagram
\[
\begin{gathered}
\xymatrix{ & C\ar[d]\\
X\ar[r]^f \ar@{-->}[ur] & Y\ar[d]_s\\
& Y^\prime
}
\end{gathered} \ .
\]
Thus the dashed arrow exists by the exactness of $\Hom_{\slmod{H_n}}(X,\mbox{-})$ applied to the distinguished triangle $C\rightarrow Y\stackrel{s}{\rightarrow} Y' \stackrel{[1]}{\rightarrow} C[1]$.
\end{remark}

 A \emph{standard distinguished triangle} in $\mathcal{O}_n$ is the image of a distinguished triangle in $\slmod{H_n}$, and any triangle of $\mathcal{O}_n$ isomorphic to a standard distinguished triangle is
called a \emph{distinguished triangle}.

\subsection{Tensor triangulated structure}
Our goal in this part is to establish the triangulated tensor category structure on $\cO_n$ which is inherited from that of $\slmod{H_n}$ under localization.

\begin{lemma}\label{lem-triang-structure}
The following functors on $\slmod{H_n}$ descend to (bi-)exact functors on $\mathcal{O}_n$:
\begin{enumerate}
\item[(1)] The tensor product $(\mbox{-}\otimes \mbox{-})\colon \slmod{H_n}\times \slmod{H_n} \longrightarrow \slmod{H_n}$.\item[(2)] The inner hom $\Hom^\bullet (\mbox{-}, \mbox{-})\colon \slmod{H_n}^{\mathrm{op}}\times \slmod{H_n} \longrightarrow \slmod{H_n}$.
\item[(3)] The grading shift functors $\shift{k}\colon  \slmod{H_n}\longrightarrow \slmod{H_n}$, where $k\in \mathbb{Z}$.
\item[(4)] The vector space dual $(\mbox{-})^*:\slmod{H_n}\longrightarrow \slmod{H_n}$.
\end{enumerate}
\end{lemma}
\begin{proof}
The tensor product functor $\otimes$ on $\slmod{H_n}$ is bi-exact \cite{Kh}. Thus, for (1), it suffices to show that it preserves the class of quasi-isomorphisms. Let $s\colon  M\rightarrow M^\prime$ be a quasi-isomorphism in $\slmod{H_n}$ that arises from an actual $H_n$-module map $s\colon M\rightarrow M^\prime$. Replacing $s$ by $(s,\rho_M)\colon M\longrightarrow M^\prime\oplus M\otimes H\shift{\l}$ if necessary, we may assume from the start that $s$ is injective. Thus $C:=\mathrm{coker}(s)$ is isomorphic to a module in $\bfI$ in $\slmod{H_n}$, and a direct sum of $C$ by some projective-injective $H_n$-module belongs to $\bfI$. Since $\bfI$ is closed under summands (Lemma \ref{lem-summands-closed}), we may assume $C$ is also in $\bfI$. Tensoring the exact sequence
\[
0\longrightarrow M \stackrel{s}{\longrightarrow} M^\prime \longrightarrow C\longrightarrow 0
\]
with any module $N$ on the left, we have a short exact sequence
\[
0\longrightarrow N\otimes M \xrightarrow{\mathrm{Id}_N\otimes s} N\otimes M^\prime \longrightarrow N\otimes C\longrightarrow 0.
\]
By Lemma \ref{lem-ext-closed}, $N\otimes C\in \bfI$, and hence $\mathrm{Id}_N\otimes s$ descends to a quasi-isomorphism in $\slmod{H_n}$. The case of tensoring on the right is similar, and this finishes the proof of (1).

Part (4) is clear since the dual of any object in $\bfI$ is also in $\bfI$ by definition.
Now part (2) and (3) are easy consequences of (1) and (4) because of Corollary \ref{gradingshift} and the isomorphism $\Hom^\bullet(M,N)\cong M^*\otimes N$ of equation \eqref{internalhom_dual}. 
\end{proof}

We are now ready to establish a tensor-hom adjunction in our category $\mathcal{O}_n$.

\begin{theorem}\label{thm-tensor-hom-adj-On}
The tensor-hom adjunction holds in $\mathcal{O}_n$:
\[
\Hom_{\mathcal{O}_n}(M\otimes L, N)\cong \Hom_{\mathcal{O}_n}(L,\Hom^\bullet(M,N)),
\]
where $M$, $N$ and $L$ are arbitrary objects of $\mathcal{O}_n$.
\end{theorem}
\begin{proof}
Given a morphism $f\in \Hom_{\mathcal{O}_n}(L, \Hom^\bullet(M,N))$ represented by a ``roof'' diagram in $\slmod{H_n}$
\begin{align*}
\begin{gathered}
\xymatrix{&L^\prime \ar[dl]_{s}\ar[dr]^{g}&\\
L\ar@{-->}[rr]^-{f}&&\Hom^\bullet (M,N)
}
\end{gathered} \ ,
\end{align*}
we have, by the adjunction \eqref{eqn-tensor-hom-adj}, another ``roof'' $f^\prime\in\Hom_{\mathcal{O}_n}(M\otimes L, N)$
\begin{align*}
\begin{gathered}
\xymatrix{&M\otimes L^\prime \ar[dl]_{\mathrm{Id}_M\otimes s}\ar[dr]^{g^{\prime}}&\\
M\otimes L\ar@{-->}[rr]^-{f^\prime}&&N
}
\end{gathered} 
\end{align*}
since $\mathrm{Id}_M\otimes s$ is a quasi-isomorphism of degree zero (see the proof of Lemma \ref{lem-triang-structure}). Here $g^{ \prime}$ is the degree zero map that corresponds to $g$ under the isomorphism \eqref{eqn-tensor-hom-adj}. In other words, we have constructed a map of morphism spaces
\begin{equation}\label{eqn-desired-iso}
\Hom_{\mathcal{O}_n}(M\otimes L, N)\longrightarrow \Hom_{\mathcal{O}_n}(L,\Hom^\bullet(M,N)), \quad \quad f\mapsto f^\prime,
\end{equation}
which gives rise to a natural transformation of cohomological functors
\begin{equation}\label{eqn-desired-iso2}
\Hom_{\mathcal{O}_n}(\mbox{-}\otimes L, N)\Longrightarrow \Hom_{\mathcal{O}_n}(L,\Hom^\bullet (\mbox{-},N)).
\end{equation}
Now, assume that $M$ is an actual $H_n$-module. We will prove that the natural transformation of functors \eqref{eqn-desired-iso2} is an isomorphism by induction on the dimension of $M$.

If $M$ is one-dimensional, then, up to a grading shift on $M$, we may assume that $M=\Bbbk$, and \eqref{eqn-desired-iso} reduces to an isomorphism
\[
\Hom_{\mathcal{O}_n}(\Bbbk \otimes L, N)\cong \Hom_{\mathcal{O}_n}(L,N) \cong\Hom_{\mathcal{O}_n}(L,\Hom^\bullet(\Bbbk,N)).
\]

When $\mathrm{dim}(M)>1$, we may assume, up to grading shift, that $M$ contains a copy of $\Bbbk$ in its socle. This can be done since $H_n$ is a graded local algebra. Then we have a short exact sequence of $H_n$-modules
\[
0\longrightarrow \Bbbk\longrightarrow M\longrightarrow M^\prime \longrightarrow 0,
\]
where $M^\prime$ denotes the quotient. This sequence induces a distinguished triangle in $\slmod{H_n}$ and descends to a standard distinguished triangle in $\mathcal{O}_n$. Applying \eqref{eqn-desired-iso2} to the obtained triangle in $\mathcal{O}_n$, we obtain a map of exact triangles:
\[
\begin{gathered}
\xymatrix@=1.2em{
\Ext_{\mathcal{O}_n}^\bullet(\Bbbk\otimes L, N) \ar[r]\ar[d]& \Ext_{\mathcal{O}_n}^\bullet(M\otimes L, N)\ar[r]\ar[d] & \Ext_{\mathcal{O}_n}^\bullet(M^\prime\otimes L, N) \ar[r]^-{[1]} \ar[d] & \\
\Ext_{\mathcal{O}_n}^\bullet( L,\Hom^\bullet(\Bbbk, N)) \ar[r] & \Ext_{\mathcal{O}_n}^\bullet( L,\Hom^\bullet(M, N)) \ar[r] & 
\Ext_{\mathcal{O}_n}^\bullet( L,\Hom^\bullet(M^\prime, N)) \ar[r]^-{[1]} & 
}
\end{gathered} \ .
\]
Here we have adopted the conventional notation 
$$\Ext_{\mathcal{O}_n}^\bullet(L,N):=\bigoplus_{i\in \Z}\Hom_{\mathcal{O}_n}(L,M[i]).$$
The left-most and, by inductive hypothesis, the right-most vertical arrow are isomorphisms of $\Ext$-groups. The theorem then follows from the usual ``two-out-of-three'' properties for distinguished  triangles in triangulated categories.
\end{proof}

\begin{remark}\label{rem-another-proof}
As pointed out by the referee, Theorem \ref{thm-tensor-hom-adj-On} admits a more conceptual proof than the explicit one above, as follows.

Suppose that $\mathcal{L} \colon \mathcal{C} \longrightarrow \mathcal{D}$ is a functor admitting a right adjoint $\mathcal{R}$. Let
$\Sigma_{\mathcal{C}}$ and $\Sigma_{\mathcal{D}}$ be classes of morphisms in $\mathcal{C}$ and $\mathcal{D}$ such that $\mathcal{L}(\Sigma_{\mathcal{C}})\subseteq \Sigma_{\mathcal{D}}$ and $\mathcal{R}(\Sigma_{\mathcal{D}})\subseteq \Sigma_{\mathcal{C}}$. Then it is immediate from the universal property of Verdier localization that $(\mathcal{L}, \mathcal{R})$ induces a pair
of adjoint functors between the localized categories $\mathcal{C}(\Sigma_{\mathcal{C}}^{-1})$ and  $\mathcal{D}(\Sigma_{\mathcal{D}}^{-1})$.

Now, in our situation, $\mathcal{C}=\mathcal{D}=\slmod{H_n}$. Take $\mathcal{L}$ and $\mathcal{R}$ to be the tensor and $\Hom$ functors respectively. It suffices to check that these functors preserve the ideal $\underline{\bfI}$, which in turn follows from the proof of Lemma \ref{lem-triang-structure}.
\end{remark}

Taking $L=\Bbbk$ in Theorem \ref{thm-tensor-hom-adj-On}, we obtain an isomorphism of $H_n$-modules
\[
\Hom_{\mathcal{O}_n}(M, N)\cong \Hom_{\mathcal{O}_n}(\Bbbk,\Hom^\bullet(M,N))\cong\Hom_{\mathcal{O}_n}(\Bbbk,M^*\otimes N),
\]
which gives an implicit description of the morphism spaces.

\begin{remark}
It remains an interesting question to compute the endomorphism (resp.~ $\mathrm{Ext}^{\bullet}$) algebra of the unit object $\Bbbk\in \mathcal{O}_n$. Since $\Bbbk$ is the (triangulated) monoidal unit, the endomorphism (resp.~$\mathrm{Ext}^{\bullet}$) algebra is a commutative (resp.~super) $\Bbbk$-algebra. It is nonzero since, otherwise, the object $\Bbbk$ would be in $\underline{\bfI}$. This is clearly false since $\Bbbk$ is not free as a module over $\hat{H}_n^k$ for any $k=1,\dots, t$.
\end{remark}

\begin{proposition}\label{prop-shift-symmetry}
The tensor product on $\cO_n$ is compatible with homological shift in the sense that for any objects $X$ and $Y$ there are natural isomorphisms
$$(X\otimes Y)[1]\cong X[1]\otimes Y\cong X\otimes Y[1].$$
\end{proposition}
\begin{proof}
This is because the shift functor can be realized as
$$M[1]\cong M\otimes (H_n/\Bbbk \Lambda)\shift{\l}\cong (H_n/\Bbbk \Lambda)\shift{\l}\otimes M.$$ 
See \cite{Kh}*{Lemma~2} for an explicit formula of the second isomorphism. 
\end{proof}

\subsection{Rings of cyclotomic integers}
In this subsection, we prove that the Grothendieck ring of the quotient category $\mathcal{O}_n$ is isomorphic to the cyclotomic ring $\mathbb{O}_n$ of a primitive $n$th root of unity.
 
For a formal variable $\nu $, recall the notation 
$$[n]_\nu:=\dfrac{\nu^n-1}{\nu-1}=1+\nu+\cdots + \nu^{n-1}\in \mZ[\nu],$$ and denote the $n$th cyclotomic polynomial by $\Phi_n(\nu)$. We will use the following elementary facts about cyclotomic polynomials.

\begin{lemma}\label{cyclotomiclemma}Let $n=p_1^{n_1}\ldots p_t^{n_t}$, where $n_k\geq 1$ are integers, and $p_k$ are pairwise distinct primes, and $m=p_1\ldots p_t$ be the radical of $n$. Then the cyclotomic polynomials in a formal variable $\nu$ satisfy
\begin{gather}
\Phi_m(\nu)=\gcd \left([m]_\nu/[m/p_1]_\nu,\ldots,[m]_\nu/[m/p_t]_\nu\right),\label{cyclotomiceq1}\\
\Phi_{p_k}(\nu^{m/p_k})=[m]_{\nu}/[m/p_k]_{\nu},\quad \quad k=1,\ldots, t,\label{cyclotomiceq2}\\
\Phi_n(\nu)=\Phi_{m}(\nu^{n/m}).\label{cyclotomiceq3}
\end{gather}
\end{lemma}
\begin{proof}
We use the following readily verified formulas
\begin{align*}
\prod_{d|m,d>1}\Phi_d(\nu)&=[m]_{\nu},&\prod_{p_k|d,d|m}\Phi_d(\nu)&=\frac{[m]_{\nu}}{[m/p_k]_{\nu}}.
\end{align*}
They hold because the multiplicative group of $m$th roots of unity is partitioned, by the order of the root of unity, into the divisors $d$ of $m$. The product of all $\nu-q$, where $q$ is a primitive $d$th root of unity, is equal to $\Phi_d(\nu)$. It follows that, if $d\neq m$, then $d$ is  not divisible by at least one of the distinct primes $p_k$, and thus $\Phi_d(\nu)$ does not divide $[m]_\nu/[m/p_k]_\nu$. On the other hand, $\Phi_m(\nu)$ clearly divides each $[m]_\nu/[m/p_k]_\nu$, $k=1,\dots, t$. Hence the greatest common divisor of all polynomials $[m]_\nu/[m/p_k]_\nu$ is precisely $\Phi_m(\nu)$, establishing equation \eqref{cyclotomiceq1}. 

Equation \eqref{cyclotomiceq2} is easy since, for a prime $p$, $\Phi_p(\nu)=(\nu^p-1)/(\nu-1)$, so that
\[
\Phi_{p_k}(\nu^{m/p_k})=\dfrac{\nu^{m}-1}{\nu^{m/p_k}-1}=\dfrac{\frac{\nu^{m}-1}{\nu-1}}{\frac{\nu^{m/p_k}-1}{\nu-1}}=\dfrac{[m]_\nu}{[m/p_k]_\nu}.
\]
The last equation \eqref{cyclotomiceq3} is an exercise in \cite{Lang}*{Chapter IV \S 3}.
\end{proof}

In the following, we denote by $K_0(\slmod{H_n})$ the Grothendieck group of the stable category of $H_n$-modules. Given an object $V$, we denote its class in the Grothendieck group by $[V]$.  Recall that this is the abelian group generated by symbols of isomorphism classes of objects in $\slmod{H_n}$, subject to relations $[U]-[W]+[V]=0$ whenever
\[U\longrightarrow W \longrightarrow V \stackrel{[1]}{\longrightarrow} U[1]\]
is a distinguished triangle.

The monoidal structure of $H_n$ gives $K_0(\slmod{H_n})$ a ring structure, and the $\mZ$-grading shift introduced in Section \ref{Hnmodules} gives it the structure of a left and right $\mZ[\nu,\nu^{-1}]$-algebra, such that the left and right module structure coincide using the natural isomorphism from Lemma \ref{swapiso}.

\begin{lemma}
The Grothendieck group of $\slmod{H_n}$ is isomorphic, as a $\mZ[\nu,\nu^{-1}]$-algebra, to the quotient ring 
$$K_0(\slmod{H_n})\cong \dfrac{\Z[\nu,\nu^{-1}]}{(\prod_{k=1}^t\frac{[n]_\nu}{[n_k]_\nu})}.$$
The tensor product on $\slmod{H_n}$ descends to the multiplication on the Grothendieck group level, while the grading shift functor $\shift{1}$ descends to multiplication by $\nu$.
\end{lemma}
\begin{proof}
The Grothendieck ring $K_0(\slmod{H_n})$ is generated, as a $\mZ[\nu,\nu^{-1}]$-module, by the class of the only simple $H_n$-module $\Bbbk$, which is one-dimensional. The only relations imposed on the symbol of the simple module arise from graded dimensions of projective-injective $H_n$-modules. The result thus follows from Lemma \ref{indproj}.
 \end{proof}

In contrast, the Verdier quotient category $\cO_n$ categorifies the cyclotomic ring $\mathbb{O}_n$.

\begin{theorem}\label{thm-main}
The Grothendieck ring  of $\cO_n$ is isomorphic to the ring of cyclotomic integers 
$$K_0(\cO_n)\cong \dfrac{\Z[\nu,\nu^{-1}]}{(\Phi_n(\nu))}.$$
\end{theorem}
\begin{proof}
We have an exact sequence of triangulated categories
$$\un{\bfI}\hookrightarrow \slmod{H_n}\twoheadrightarrow \cO_n,$$
where the first containment is fully faithful and idempotent complete (Lemma \ref{lemma-thickness}). It follows from well-known facts on $K$-theory of exact sequence of triangulated categories that
$$K_0(\cO_n)=K_0(\slmod{H_n}/\un{\bfI})=K_0(\slmod{H_n})/K_0(\un{\bfI})$$
(see, for instance, \cite{Schl}*{3.1.6}).
We will determine the image $I:=K_0(\un{\bfI})$ in $K_0(\slmod{H_n})$. Note that $I$ is an ideal in the ring $K_0(\slmod{H_n})$ by Lemma \ref{lem-triang-structure}, generated by the classes $\Gro{V}$ for all objects $V$ in $\un{\bfI}$. 

Let $\nu$ be the formal variable representing the image in $K_0(\cO_n)$ of the object $\Bbbk\shift{1}$ of $\mathcal{O}_n$. Write $\mu:=\nu^{n/m}$ and $\mu_k:=\nu^{n/p_k}$.

By definition, any object $U\in\underline{\bfI}$ is isomorphic to a module $U'\in \lmod{H_n}$ 
such that $U'\cong \oplus_{k=1}^t U_k$, where $U_k$ is an object in $\bfI_k$. Hence, $\Gro{U}=\sum_{k=1}^t\Gro{U_k}$. However, any object in $\bfI_k$ is, in particular, a free $\widehat{H}_n^k$-module. 
Therefore, in $K_0(\slmod{H_n})$, we have that $\Gro{U_k}$ is a $\mZ[\nu,\nu^{-1}]$-multiple of $\Gro{W_k}$. In the presence of at least two distinct prime factors $p_k,p_l$, $\Gro{V_l}$ divides $\Gro{W_k}$. Hence, the symbol $\Gro{U}$ of any object of $U$ in $\underline{\bfI}$ is
a $\mZ[\nu,\nu^{-1}]$-linear combination of the cyclotomic polynomials
\begin{align*}
\Phi_{p_k}(\mu_k)&=\Gro{V_k}=1+\mu_k+\ldots+\mu_k^{p_k-1}, &\text{for }k=1,\ldots,t.
\end{align*}

Conversely, the relations $$\Gro{W_k}=\prod_{l\neq k}[m]_\mu/[m/p_l]_\mu=0$$ hold  in $K_0(\cO_n),$ using equation \eqref{cyclotomiceq2} of Lemma \ref{cyclotomiclemma}. Therefore, the relations
$$\gcd \left\lbrace [W_l]~\middle\vert~ l=1,\ldots, t \text{ such that } l \neq k \right\rbrace=\Gro{V_k}=\frac{[m]_\mu}{[m/p_k]_\mu}=1+\mu_k+\ldots+\mu_k^{p_k-1}=0$$
are satisfied in $K_0(\cO_n)$ and generate the ideal $I$. 
Now, by equations \eqref{cyclotomiceq1} and \eqref{cyclotomiceq3}, we see that
$$\Phi_n(\nu)=\Phi_{m}(\mu)=\gcd\left( [m]_\mu/[m/p_1]_\mu,\ldots ,[m]_\mu/[m/p_t]_\mu \right)$$ generates $I$. The result follows.
\end{proof}

\begin{remark}
The theorem can be summarized as saying that the tensor triangulated category $\mathcal{O}_n$ categorifies the cyclotomic ring of integers $\mathbb{O}_n$. Choose an embedding of $\mathbb{O}_n$ in $\mathbb{C}$. The tensor product on $\mathcal{O}_n$ descends to the product of cyclotomic integers. Furthermore, the vector space dual functor $(\mbox{-})^*:\mathcal{O}_n\rightarrow \mathcal{O}_n$ decategorifies to the complex conjugation map $[M^*]=\overline{[M]}$. It also follows that the inner hom measures the complex norm of the symbols  
$$[\Hom^\bullet(M,M)]=[M^*\otimes M]=\overline{[M]}[M]=|[M]|^2.$$
\end{remark}



\bibliography{biblio}
\bibliographystyle{amsrefs}

\end{document}